\documentclass[11pt,twoside,a4paper]{article}
\usepackage[english]{babel}
\usepackage{amsmath,amssymb,amsthm,amscd,bbold,epsfig,eucal,enumitem,color}
\usepackage{graphicx}  
\usepackage{appendix}
\usepackage{wrapfig}
\usepackage{makeidx}

\DeclareMathOperator*{\argmax}{arg\,max}
\DeclareMathOperator*{\argmin}{arg\,min}

\begin{document}

\numberwithin{equation}{section}

\theoremstyle{plain}
\newtheorem{theorem}{Theorem}[section]
\newtheorem{lemma}[theorem]{Lemma}
\newtheorem{proposition}[theorem]{Proposition}
\newtheorem{corollary}[theorem]{Corollary}
\newtheorem{assumption}[theorem]{Assumptions}
\newtheorem{mainAssumption}[theorem]{Definition}
\newtheorem{reducedAssumption}[theorem]{Definition}

\theoremstyle{definition}
\newtheorem{remark}[theorem]{Remark}
\newtheorem{definition}[theorem]{Definition}
\newtheorem{notation}[theorem]{Notation}

\newcommand{\keywords}[1]{\par\addvspace\baselineskip\noindent\textbf{Keywords:}\enspace\ignorespaces#1}
\newcommand{\AMSclassification}[1]{\par\addvspace\baselineskip\noindent\textbf{2010 Mathematics Subject Classification:}\enspace\ignorespaces#1}

\newcommand{\osc}{\text{\rm var}}
\newcommand{\supp}{\text{\rm supp}}
\newcommand{\card}{\text{\rm card}}

\title{Zero-temperature phase diagram for double-well type potentials in the summable variation class}
\author{
\small{Rodrigo Bissacot}\thanks{Supported by CNPq grant 308583/2012-4 and FAPESP grant 11/16265-8.
Supported also by EU Marie-Curie IRSES Brazilian-European partnership in Dynamical Systems (FP7-PEOPLE-2012-IRSES 318999 BREUDS).} \\	 
\footnotesize{Department of Applied Mathematics}\\
\footnotesize{University of Sao Paulo}\\
\footnotesize{05508-090 - Sao Paulo, Brazil}\\
\footnotesize{\texttt{rodrigo.bissacot@gmail.com}}
\and
\small{Eduardo Garibaldi}\thanks{Supported by CNPq Universal 476562/2013-9, CNPq grant 308593/2014-6 and VRERI-UNICAMP.}\\
\footnotesize{Department of Mathematics}\\
\footnotesize{University of Campinas}\\
\footnotesize{13083-859 Campinas, Brazil}\\
\footnotesize{\texttt{garibaldi@ime.unicamp.br}}
\and
\small{Philippe Thieullen}\thanks{Supported by FAPESP 15/10398-7. 
Supported also by EU Marie-Curie IRSES Brazilian-European partnership in Dynamical Systems (FP7-PEOPLE-2012-IRSES 318999 BREUDS).}\\
\footnotesize{Institut de Math\'ematiques}\\
\footnotesize{Universit\'e de  Bordeaux, CNRS, UMR 5251}\\
\footnotesize{F-33405 Talence, France}\\
\footnotesize{\texttt{Philippe.Thieullen@math.u-bordeaux1.fr}}
}
\date{\today}

\maketitle

\begin{abstract}
We study the zero-temperature limit of the Gibbs measures of a class of long-range potentials on a full shift of two symbols $\{0,1\}$. 
These potentials were introduced by Walters as a natural space for the transfer operator. 
In our case, they are constant on a countable infinity of cylinders, and Lipschitz continuous or, more generally, of summable variation. 
We assume there exists exactly two ground states: the fixed points $0^\infty$ and $1^\infty$. 
We fully characterize, in terms of the Peierls barrier between the two ground states, the zero-temperature phase diagram of such potentials,
that is, the regions of convergence or divergence of the Gibbs measures as the temperature goes to zero.
\end{abstract}

\section{Introduction and main results}

We consider the problem of convergence or divergence of Gibbs measures as the absolute temperature goes to zero. 
By a Gibbs measure, we mean an invariant probability $\mu_\beta$ describing the equilibrium at temperature $\beta^{-1}$  of one-sided 
configurations $(x_0,x_1,\ldots) \in \Sigma := \{0,1\}^{\mathbb{N}}$ interacting according to a potential $H : \Sigma \to \mathbb{R}$ as described in the thermodynamic formalism (see \cite{Baladi2000_01,Keller,ParryPollicott1990_01,Ruelle}). 
The invariance of the measure is defined with respect to the left shift $\sigma : \Sigma \to \Sigma$,
$ \sigma(x_0, x_1, \ldots) = (x_1, x_2, \ldots) $. 
We assume in the following that $H$ is nonnegative, Lipschitz continuous, or more generally of summable variation. When $\beta \to +\infty$, the Gibbs measures tend 
to concentrate on the minima of $H$. Besides, the limit measure needs to be invariant. We assume that  the only invariant ergodic probability measures included in the zero-level set $\{H=0\}$ are exactly the two Dirac measures $\delta_{0^\infty}$ and  $\delta_{1^\infty}$. 
As the temperature goes to zero ($\beta \to +\infty$), two cases may happen, either the {\it selection case} where $\mu_\beta$ converges to a convex combination  $c_0\delta_{0^\infty} + c_1\delta_{1^\infty}$, or the {\it nonselection case} where, for some subsequence $\beta_k$,
$ \{\mu_{\beta_k}\} $ has two accumulation points: $\mu_{\beta_{2k}} \to \delta_{0^\infty}$ and $\mu_{\beta_{2k+1}} \to \delta_{1^\infty}$. We  consider in this work the smallest class of potentials where the two cases coexist. 

For potentials that depend on a finite number of coordinates, namely, that are constant on a finite number of cylinder sets, the selection case always holds, over both finite alphabets
\cite{Bremont2003, Leplaideur2005, ChazottesGambaudoUgalde2010, GaribaldiThieullen2012} and countably infinite alphabets \cite{Kempton, FreireVargas}. 
For potentials that are constant on a countable infinity of cylinders, the selection case has been proved in particular examples: see Baraviera, Leplaideur, Lopes in~\cite{BaravieraLeplaideurLopes2012}, Leplaideur in~\cite{Leplaideur2012}, Baraviera, Lopes, Mengue in~\cite{BaravieraLopesMengue2013_01}. The nonselection case has been addressed more recently in \cite{VanEnterRusze2007l}, \cite{ChazottesHochman2010} and \cite{CoronelRiveraLetelier2014}. In a seminal paper~\cite{VanEnterRusze2007l}, van Enter and Ruszel have produced  an example 
where {\it chaotic temperature dependence} was observed, however their alphabet is the unit circle and the construction
is only based on properties of the potential and not on the dynamics. 
Chazottes and Hochman gave in \cite{ChazottesHochman2010} examples of nonselection in any dimension $D\not=2$ (with respect to an underlying 
$\mathbb Z^D$-action).
In one dimension, their potential is equal to the distance to some invariant compact set that has a complex 
combinatorial construction. In dimension $ D \ge 3 $, their nonselection examples come from potentials that do depend on a finite number of coordinates. Recently in~\cite{AubrunSablik}, Aubrun and Sablik extended \cite{Hochman2009}, which is the main ingredient 
in the proof of the multidimensional part of~\cite{ChazottesHochman2010}. In principle, an analogous proof of the nonselection 
for $D=2$ should also work. 
In \cite{CoronelRiveraLetelier2014}, Coronel and Rivera-Letelier adapted for finite alphabets van Enter and  Ruszel ideas and
they ensure the existence of nonselection examples by a perturbative approach combined with entropy arguments 
as in~\cite{ChazottesHochman2010}. Moreover, they were able to verify the nonselection case also for $D=2$, without using 
the result of~\cite{AubrunSablik}, but with Lipschitz continuous potentials. Thus, for potentials 
that depend on a finite number of coordinates in dimension $ D = 2 $,  it is an open question whether there exist examples of nonselection.

Our approach is different. We highlight the simplest class of potentials 
whose zero-temperature phase diagram is completely understood: it contains both the nonselection and the selection cases, with an 
explicit description of the limit measures in the convergent situation. 
We show that the criterion of nonselection or selection is given by the fact that the Peierls barriers between the two 
configurations $0^\infty$ and $1^\infty$ are both equal to zero or not.

We now detail such a class of potentials.
A cylinder of length $n\geq1$ is a set $C_n := [i_0i_1\ldots i_{n-1}]$ of configurations $x\in\Sigma$ such that the first 
$n$ states $x_0,x_1,\ldots,x_{n-1}$ coincide with $i_0,i_1,\ldots,i_{n-1}$. We say that two points $x,y \in \Sigma$ are $n$-close, 
and we write $x \stackrel{n}{=} y$, if $x$ and $y$ belong to the same cylinder of length $n$.
Let $H : \Sigma \to \mathbb{R}$ be a $C^0$ nonnegative potential. We say that $H$ has summable variation if 
\begin{equation}
\sum_{n\geq1} \osc(H,n) < +\infty,  \quad \text{ with } \, \osc(H,n) := \sup \big \{ | H(x) - H(y)| \,:\, x \stackrel{n}{=} y \big \}.
\end{equation}
We restrict the potential $H$ to a subclass of functions that are constant on a countable infinity of cylinders as described in the following assumptions. Our subclass  is a particular class of Walters potentials with summable variation (see \cite{Walters2007_01}).

\begin{mainAssumption} \label{assumption:locallyConstant}
We say that $H$ is a double-well type potential if $H$ is nonnegative, has summable variation and is constant on the cylinders $[00^n1]$, $[01^n0]$, $[11^n0]$ and $[10^n1]$. More precisely, there are nonnegative sequences $ \{a_n^0\} $, $ \{a_n^1\} $ and strictly positive 
sequences $ \{b_n^0\} $, $ \{b_n^1\} $ such that
\begin{enumerate}
\item \label{assumption:locallyConstant_1}  $H(x)=a_n^0 \geq 0$, \  if  $x \in [00^{n}1]$, \quad $H(x)=a_n^1 \geq 0$, \  if $ x \in [11^{n}0]$;

\item \label{assumption:locallyConstant_2}  $H(x) = b_n^0 > 0$, \ if $x \in [01^n0]$, \quad   $H(x)=b_n^1 > 0$, \ if $x\in [10^n1]$;

\item \label{assumption:locallyConstant_3} $\sum_{n\geq1} na_n^{0} < + \infty, \quad \sum_{n\geq1} n a_n^{1} < +\infty$;

\item \label{assumption:locallyConstant_4} $\sum_{k\geq1}\sup_{n\geq0}|b_k^0 - b_{k+n}^0| < +\infty$, 
\quad  $\sum_{k\geq1} \sup_{n\geq0}|b_k^1 - b_{k+n}^1| < +\infty$.
\end{enumerate}
Denote 
\begin{align*}
&H_{min}^0 := \inf_{n\geq1} \Big\{ b_n^0 + \sum_{k=1}^{n-1}a_k^1 \Big\}, \qquad 
H_\infty^0 := \lim_{n\to+\infty} b_n^0 + \sum_{n\geq1} a_n^1, \\
&H_{min}^1 := \inf_{n\geq1} \Big\{ b_n^1 + \sum_{k=1}^{n-1}a_k^0 \Big\}, 
\qquad H_\infty^1 := \lim_{n\to+\infty}  b_n^1 + \sum_{n\geq1} a_n^0.
\end{align*}
\end{mainAssumption}

As example of a double-well type potential, consider $ H : \Sigma \to [0, +\infty) $ given by $ H(0^\infty) = 0 = H(1^\infty) $
and $ H(x) = \rho_0^{\theta_0(x)} \rho_1^{\theta_1(x)} $ if $ x $ is not a fixed point, where $ \rho_0, \rho_1 \in (0,1) $
and $ \theta_0, \theta_1 \ge 1 $ are functions such that their restrictions $ \theta_0 |_{[1]} $, $ \theta_1 |_{[0]} $, 
$ \theta_0 |_{[0^n 1]} $, $ \theta_1 |_{[1^n 0]} $ are identically constant and satisfy 
$ \inf_{n \ge 1} \{ \theta_0 |_{[0^{n+1} 1]} - \theta_0 |_{[0^n 1]}, \theta_1 |_{[1^{n+1} 0]} - \theta_1 |_{[1^n 0]} \} > 0 $.
For this particular example, Gibbs measures do converge when the system is frozen as follows from our main result. 

Our main theorem describes the zero-temperature phase diagram of double-well type potentials (see figure~\ref{figure:diagram}). 
The different regions of the diagram  are described by a unique parameter, 
obtained by taking the minimum of three exponents:
\begin{equation}
\gamma := \min \Big\{\frac{1}{2} \big( H_\infty^1 + H_\infty^0 \big) , \, H_{min}^0 + H_\infty^1, \, H_{min}^1 + H_\infty^0\Big\}. \label{equation:threeExponents}
\end{equation}
Notice that $\gamma=0$ if, and only if, $H_\infty^0 = H_\infty^1=0$ if, and only if, the three exponents coincide. By symmetry we may assume $H_\infty^0 \leq H_\infty^1$. We  state the theorem in this case. If $\gamma > 0$, one exponent is irrelevant and we have: 
\[
\gamma = \min \Big\{\frac{1}{2} \big( H_\infty^1 + H_\infty^0 \big) , \, H_{min}^1 + H_\infty^0 \Big\},
\]
since $\frac{1}{2} ( H_\infty^1 + H_\infty^0)  <  H_{min}^0 + H_\infty^1$. We introduce in that case the coincidence number $\kappa$ which counts how many times the minimum is attained, that is, for $ H_n^1 := b_n^1 + \sum_{k=1}^{n-1}a_k^0 $,
\begin{equation}
   \kappa := \card \Big\{ n \geq 1 \,:\,  \frac{1}{2} \big( H_\infty^1 + H_\infty^0 \big) =  H_n^1 + H_\infty^0 \Big\}, \label{equation:coincidenceNumber} 
\end{equation}
and a coefficient $c$,  the  largest solution of the equation $ X^2 = \kappa X+1$,
\begin{equation} 
c := \frac{\kappa + \sqrt{\kappa^2+4}}{2}. \label{equation:PuiseuxCoefficient}
\end{equation}
Our main theorem is thus stated as follows.

\begin{theorem} \label{theorem:main}
Let $H:\Sigma \to \mathbb{R}$ be a double-well type potential. Let $\mu_\beta$ be the Gibbs measure of $H$ at temperature $\beta^{-1}$. 
Assume that $H_\infty^0 \leq H_\infty^1$. 

\begin{enumerate}
\item \label{item:mainTheorem_0} If $ \frac{1}{2} ( H_\infty^1 + H_\infty^0) > H_{min}^1 + H_\infty^0$, \ then \  
$ \displaystyle{\lim_{\beta\to+\infty} \mu_\beta = \delta_{1^\infty}}$.
\item \label{item:mainTheorem_1}  If $ H_{min}^1 + H_\infty^0 \geq \frac{1}{2} ( H_\infty^1 + H_\infty^0) > 0$, \ then
\begin{equation}
\lim_{\beta \to + \infty }\mu_\beta = \frac{1}{1+c^2}\delta_{0^\infty} + \frac{c^2}{1+c^2} \delta_{1^\infty}. \label{equation:SelectionCase}
\end{equation}

\item \label{item:mainTheorem_2} If $H_\infty^0=H_\infty^1=0$, then there exists a particular choice of $b_n^0,b_n^1$ (necessarily $a_n^0=a_n^1=0$)  such that  $H$ is Lipschitz and $\mu_\beta$ does not converge. More precisely, there exists a sequence $\beta_k \to +\infty$ such that$\displaystyle{\lim_{k\to+\infty}\mu_{\beta_{2k}} = \delta_{0^\infty}}$ and 
$\displaystyle{\lim_{k\to+\infty} \mu_{\beta_{2k+1}} = \delta_{1^\infty}}$.
\end{enumerate}
(Items \ref{item:mainTheorem_0} and \ref{item:mainTheorem_1} correspond to $\gamma>0$; item \ref{item:mainTheorem_2} corresponds  to $\gamma=0$.)
\end{theorem}

\begin{figure}[hbt]
\centering
\includegraphics[width=0.95\textwidth]{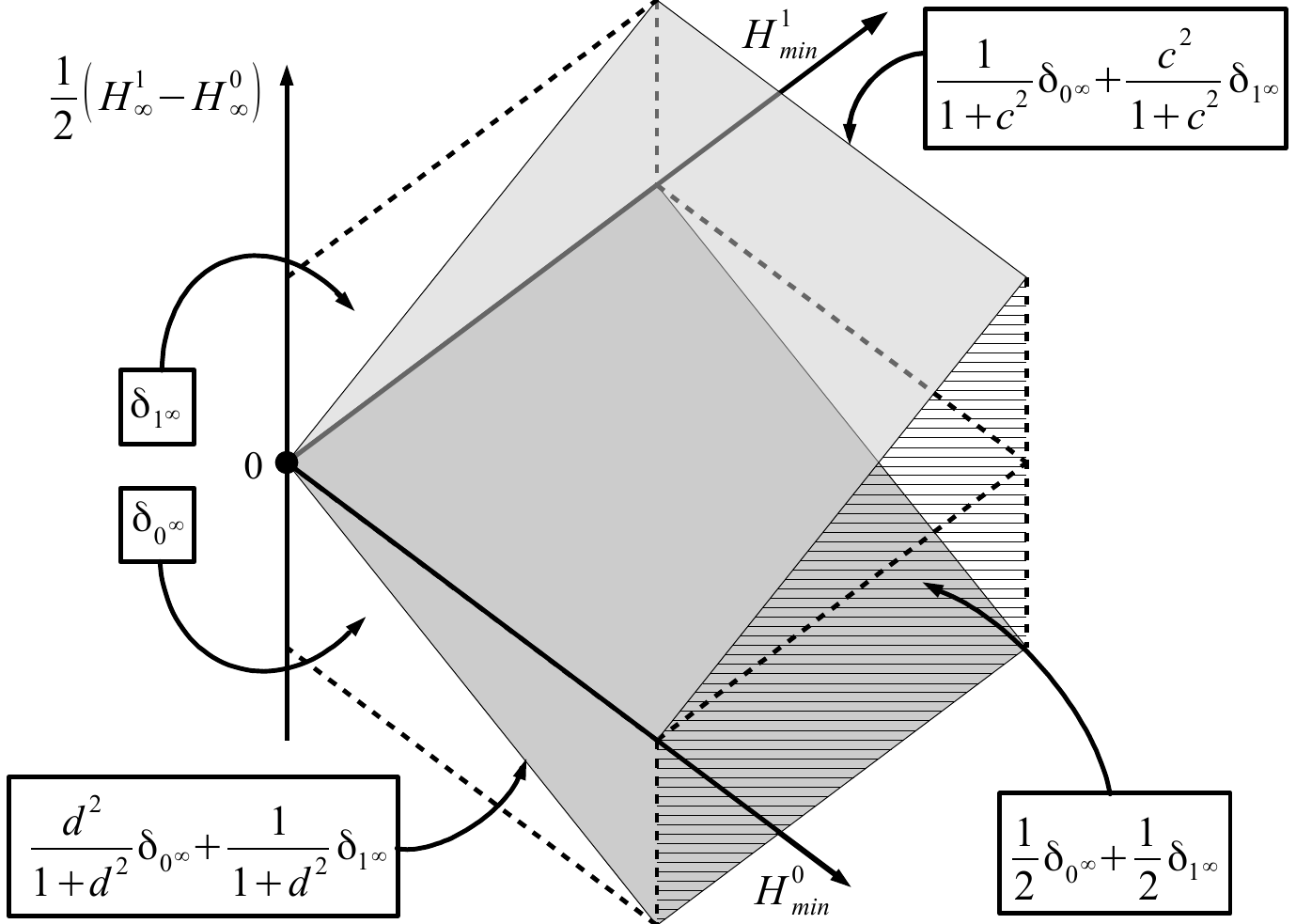}
\caption{\footnotesize{\textbf{Zero-temperature phase diagram.} The nonselection case can occur only at the origin. The formulas in the boxes
are the limit measures at zero temperature. The two gray planes correspond to the cases of the coincidence of two exponents. Outside these
planes the limit measures are barycenters with rational coefficients. If $ H_\infty^1 \ge H_\infty^0 $, then $ c $ is the coefficient
given by~\eqref{equation:PuiseuxCoefficient}. If $ H_\infty^0 \ge H_\infty^1 $, then $ d $ is the analogous coefficient.}}
\label{figure:diagram}
\end{figure}

In section~\ref{section:generalTheory}, we give general results for potentials of summable variation. In section~\ref{section:explicitFomulas}, for a double-well type potential $ H $, we compute the measure of every cylinder using two series that capture all the complexity of the limit. 
In section~\ref{section:selectionCase}, we prove the convergence of Gibbs measures when $\gamma>0$. Finally, 
in section~\ref{section:nonSelectionCase}, we provide examples of divergence with $\gamma=0$. Note that the symmetric case $a_n^0=a_n^1$ and $b_n^0=b_n^1$ gives in both cases $\gamma>0$ or $\gamma=0$ the convergence to $\frac{1}{2} \delta_{0^\infty} + \frac{1}{2} \delta_{1^\infty}$.

We also show in this particular class of potentials that the dichotomy selection/nonselection in theorem \ref{theorem:main} can be expressed in terms of the Peierls barrier between the two configurations $0^\infty$ and $1^\infty$. The Peierls barrier is defined for any potential with summable variation by
\begin{align*}
&h(x,y) := \lim_{p\to + \infty} \lim_{n\to+\infty} S_n^p(x,y), \quad \text{where}\\
&S_n^p(x,y) := \inf \Big\{ \sum_{i=0}^{k-1} \big[ H \circ \sigma^i(z)-\bar H \big] \,:\, k\geq n, \  z\in\Sigma, \  z \stackrel{p}{=} x, \ \sigma^n(z) \stackrel{p}{=} y  \Big\}, \\
& \bar H := \lim_{n\to+\infty} \ \inf \ \Big\{ \frac{1}{n} \sum_{k=0}^{n-1}H \circ \sigma^k(x) \,:\, x\in\Sigma \Big\}.
\end{align*}
The Peierls barrier indicates the minimal algebraic cost from $x$ to $y$ using a normalized potential $H - \bar H$. 
In the particular case of double-well type potentials, we have the following result.

\begin{corollary}
Let $H$ be a double-well type potential. Then
\begin{enumerate}
\item $\frac{1}{2}(H_\infty^0 + H_\infty^1) = \frac{1}{2} \big( h(0^\infty,1^\infty)+ h(1^\infty, 0^\infty) \big)$;
\item $H_{min}^0+H_\infty^1 = \liminf_{x \to 0^\infty} h(x,0^\infty)$;
\item $H_{min}^1+H_\infty^0 = \liminf_{x \to 1^\infty} h(x,1^\infty)$;
\item the nonselection happens if, and only if, $h(0^\infty,1^\infty)= h(1^\infty,0^\infty) = 0 $.
\end{enumerate}
\end{corollary}

Note that $\gamma$ may be seen as the minimum of three energy barriers: 
$\frac{1}{2} \big( H_\infty^0 + H_\infty^1 \big)$, the mean energy barrier of a cycle of second order 
between the two ground states $0^\infty$ and $1^\infty$; 
$H_{min}^0+H_\infty^1$, the energy barrier of a cycle of first order at $0^\infty$; 
and $H_{min}^1+H_\infty^0$, a similar energy barrier at $1^\infty$.

\section{Basic facts for potentials of summable variation}
\label{section:generalTheory}

We gather in this section some of the main elements of ergodic optimization theory for potentials of summable variation.
Ergodic optimization may be seen as a counterpart at zero temperature of thermodynamic formalism. 
A useful viewpoint on ergodic optimization is provided by Aubry-Mather theory. 
For more information, we refer the reader, for instance, to \cite{GLT, GaribaldiThieullen2012} and the references therein.

\begin{definition} \label{definition:minimizingMeasure}
For $H \in C^0(\Sigma)$, a minimizing measure $\mu_{min}$ is a $\sigma$-invariant probability such that
\begin{gather*}
\int\! H \,d\mu_{min} = \min \Big\{ \int\! H \,d\nu \,:\, \nu \text{ is a $\sigma$-invariant probability  measure} \Big\}.
\end{gather*}
We call Mather set of $H$ the invariant compact set
\begin{gather*}
\text{\rm Mather}(H) := \bigcup \, \{ \supp(\mu) \,:\, \mu \text{ is minimizing} \}.
\end{gather*}
We call minimizing ergodic value of $H$ the constant
\begin{gather*}
\bar H := \int\! H \,d\mu_{min}.
\end{gather*}
\end{definition}

We recall or extend basic results about the Peierls barrier for functions with summable variation.

\begin{proposition}\label{inclusao Mather}
If $H$ has summable variation, then
\begin{equation}\label{Mather incluso barreira}
\text{\rm Mather}(H) \subset \{ x \in \Sigma \,:\, h(x,x)=0 \}.
\end{equation}
\end{proposition}

The previous proposition follows from Atkinson's theorem \cite{Atkinson} and from the existence of a continuous calibrated sub-action.

\begin{definition} \label{definition:DiscreteLaxOleinik}
We call  Lax-Oleinik operator the nonlinear operator acting on continuous functions $V \in C^0(\Sigma)$ defined by
\begin{gather*}
T[V](y) := \min \{ V(x)+H(x) \,:\, x \in \Sigma, \ \sigma(x)=y \}, \quad \forall \, y \in \Sigma.
\end{gather*}
We call calibrated sub-action any continuous function  $V$ solution of the equation $T[V]=V+\bar H$. 
\end{definition}

Clearly, $ V \circ \sigma - V \le H - \bar H $ when $ V $ is a calibrated sub-action, which in particular ensures that $ h(x,x) \ge 0 $ for all
$ x \in \Sigma $. Atkinson's theorem provides the opposite inequality if $ x \in \text{\rm Mather}(H) $. These are the main ingredients of the
proof of proposition~\ref{inclusao Mather}.
To obtain a calibrated sub-action, we will introduce a stronger notion of regularity on $C^0(\Sigma)$. Consider thus
\begin{equation*}
\mathbb{K} := \Big\{ V  \in C^0(\Sigma) \,:\, \forall \, n\geq1, \  \osc(V,n) \leq \sum_{k\geq n+1} \osc(H,k) \Big\}.
\end{equation*}
We also recall that the transfer operator is defined on the space $C^0(\Sigma) $ by
\[
\mathcal{L}_\beta[\Phi] (x ) = e^{-\beta H(0x)} \Phi(0x) + e^{-\beta H(1x)} \Phi(1x), \qquad \forall \, x \in \Sigma.
\]
The next theorem contains a version of Ruelle-Perron-Frobenius theorem and provides a calibrated sub-action 
in the context of potentials with sum\-ma\-ble variation, making explicit well-known connections between thermodynamic formalism and
ergodic theory.

\begin{theorem} \label{theorem:RuelleOperator}
Let $H:\Sigma \to \mathbb{R}$ be a potential with summable variation. 
\begin{enumerate} 
\item The transfer operator admits a unique positive and continuous eigenfunction $\Phi_\beta$ satisfying $\max \Phi_\beta =1$, which is
associated with a positive eigenvalue $ \lambda_\beta $.

\item If $V_\beta := -\frac{1}{\beta} \ln \Phi_\beta$, then $V_\beta \in \mathbb{K}$ and $\min V_\beta =0$. 

\item The dual operator $\mathcal{L}_\beta^*$ admits a unique eigenprobability $\nu_\beta$. 
The corresponding eigenvalue is equal to $\lambda_\beta$, $\mathcal{L}_\beta^*[\nu_\beta] = \lambda_\beta \nu_\beta$. 

\item Define $\mu_\beta := {\Phi_\beta \nu_\beta}/{\int\! \Phi_\beta\, d\nu_\beta}$. Then $\mu_\beta$ is a $\sigma$-invariant probability measure, and any weak${^*}$ accumulation point of $\mu_\beta$ as $\beta \to +\infty$ is a minimizing measure.

\item There exists a sequence $\beta_k \to +\infty$ such that (in the sup-norm topology) $\{V_{\beta_k}\} $ converges to a function $ V_\infty \in \mathbb{K}$ with $\min V_\infty = 0$. Moreover, any accumulation function $V_\infty$ of $\{V_\beta\}$ as $\beta \to +\infty $ is a calibrated sub-action for $ H $.
\end{enumerate}
\end{theorem}

\begin{proof}
The proof of these results are standard (see \cite{Ruelle,ParryPollicott1990_01,GaribaldiThieullen2012}), 
and hence we focus on the part leading to the existence of calibrated sub-actions.
We define a nonlinear operator $T_\beta$ by
\begin{gather*}
T_\beta[u] := -\frac{1}{\beta}\ln \big( \mathcal{L}_\beta[\exp(-\beta u)] \big).
\end{gather*}
Fix $ x_0 \in \Sigma $ and define $\mathbb{K}_0 := \{ U \in \mathbb{K} \,:\, U(x_0)=0 \}$. The set $\mathbb{K}_0$ is closed in the $C^0(\Sigma)$ topology and bounded. By the unifom continuity of $\mathbb{K}$ and Arzel\`a-Ascoli theorem, the set $\mathbb{K}_0$ is compact. 
Besides, $\mathbb{K}_0$ is convex.

If $x \stackrel{n}{=} y$, then
\begin{gather*}
T_\beta[u](x) - T_\beta[u](y) \leq \osc(H,n+1) + \osc(u,n+1).
\end{gather*}
In particular $\osc(T_\beta[u],n) \leq  \osc(H,n+1) + \osc(u,n+1)$ and the map
\begin{gather*}
\tilde T_\beta[u] := T_\beta[u ] - T_\beta[u](x_0)
\end{gather*}
preserves $\mathbb{K}_0$. By Schauder theorem, $\tilde T_\beta$ admits a fixed point, or in an equivalent way, $T_\beta$ admits an additive eigenfunction $ T_\beta[U_\beta] = U_\beta + \bar H_\beta $, which yields
\begin{equation*}
\mathcal{L}_\beta[\Phi_\beta] = \lambda_\beta \Phi_\beta, \quad 
\text{with}\ \Phi_\beta := e^{-\beta (U_\beta-\min U_\beta)}, \ \lambda_\beta = e^{-\beta \bar{H}_\beta}.
\end{equation*}
Let  $\tilde \Phi$ be another positive and continuous eigenfunction associated with some positive eigenvalue $\tilde\lambda$. We choose $s,t>0$ such that $s\Phi_\beta \leq \tilde \Phi \leq t \Phi_\beta$. By iterating $\mathcal{L}_\beta$, we obtain $s \lambda_\beta^n \Phi_\beta \leq  \tilde\lambda^n \tilde\Phi \leq t \lambda_\beta^n \Phi_\beta$. Then $\tilde\lambda=\lambda_\beta$. Let $s$ be such that 
$\min(\tilde\Phi-s\Phi_\beta)=0$. Then the identity
\begin{gather*}
\mathcal{L}_\beta[\tilde\Phi -s \Phi_\beta] = \lambda_\beta(\tilde\Phi - s \Phi_\beta)
\end{gather*} 
implies that the set $\argmin_{x}  (\tilde\Phi - s\Phi_\beta )(x)$ is invariant by $\sigma^{-1}$ and therefore $\tilde\Phi = s \Phi_\beta$. The uniqueness of the eigenfunction is proved.

Note that the family $ \{V_\beta = -\frac{1}{\beta} \ln \Phi_\beta\}_{\beta>0} $ belongs to the compact subset 
$ \{ V \in \mathbb{K} \,:\, \min V = 0 \} $. Passing to the limit with respect to a suitable sequence $ \beta_k \to +\infty $,
we see that $ T[V_\infty] = V_\infty + c $ for $ c = \lim \bar H_{\beta_k} $. From min-plus algebra, it is well know that the only 
additive eigenvalue is $ c = \bar H $.
\end{proof}

The following proposition shows how calibrated sub-actions are related with the Peierls barrier.

\begin{proposition}\label{corollary:RepresentationFormula}
If $H$ has summable variation, then the following items hold. 
\begin{enumerate}
\item \label{corollary:RepresentationFormula:1} For every $x\in\text{Mather}(H)$, as a function of its second variable, 
$h(x,\cdot)$ belongs to $ \mathbb{K}$ and is a calibrated sub-action.
\item \label{corollary:RepresentationFormula:2} 
If $V \in C^0(\Sigma)$ is a calibrated sub-action, then $ V \in \mathbb K $
and $ V $ admits a representation formula\footnote{This representation is usually stated using the Aubry set instead of the Mather set.}
\begin{equation}\label{formula representacao}
V(y) = \min \big\{ V(x) + h(x,y) \,:\, x \in \text{\rm Mather}(H) \big\}, \qquad \forall \, y \in \Sigma. 
\end{equation}
\end{enumerate}
\end{proposition}

\begin{proof}
For the Lipschitz class, these results may be found in the literature
(see, for instance, \cite{GLT, GaribaldiThieullen2012} and the references therein). 
All proofs may be easily extended just adapting the arguments to the regularity here considered.
For the convenience of the reader, we outline the proofs of items~\ref{corollary:RepresentationFormula:1} 
and~\ref{corollary:RepresentationFormula:2}.  

\emph{Item~\ref{corollary:RepresentationFormula:2}.} Suppose $ y \stackrel{n}{=} z $. Denoting $ y_0 = y $, 
since $ V $ is a calibrated sub-action, there exists a sequence $ \{y_k\} \subset \Sigma $ such that
\begin{equation}\label{calibracao y0}
V(y_0) = V(y_k) + \sum_{i=1}^{k-1} [H\circ\sigma^i(y_k) - \bar H], \quad \sigma(y_k)=y_{k-1}, \quad \forall \, k \ge 1.
\end{equation} 
For $ z_0 = z $, we thus consider a sequence $ \{z_k\} $, with $ \sigma(z_k) = z_{k-1} $, such that $ z_k \stackrel{n + k}{=} y_k $ 
for all $ k $. Note that
\begin{equation}\label{subcalibracao z0}
V(z_0) \le V(z_k) + \sum_{i=1}^{k-1} [H\circ\sigma^i(z_k) - \bar H], \qquad \forall \, k \ge 1.
\end{equation} 
From~\eqref{calibracao y0} and~\eqref{subcalibracao z0}, we have $  \osc(V,n) \leq \sum_{k\geq n+1} \osc(H,k) $, that is, $ V \in \mathbb K $.

From the inequality $ V \circ \sigma - V \le H - \bar H $, given any $ y \in \Sigma $, we have that 
$ V(y) \le \min \{ V(x) + h(x,y) \,:\, x \in \text{\rm Mather}(H) \} $. For $ y_0 = y $, we consider again~\eqref{calibracao y0}.
Since
$ V(y_k) = V(y_{k+p}) + \sum_{i=1}^{p-1} [H\circ\sigma^i(y_{k+p}) - \bar H] $, for all $ k,p \ge 0 $,
one may deduce that a limit $ \bar x \in \Sigma $ of subsequence $ \{y_{k_j}\} $ satisfies $ h(\bar x, \bar x) = 0 $.
By passing to the limit in $ V(y_0) = V(y_{k_j}) + \sum_{i=1}^{k_j-1} [H\circ\sigma^i(y_{k_j}) - \bar H] $, 
we see that $ V(y) = V(\bar x) + h(\bar x,y) $. For all $ x $ in the same irreducible class as $ \bar x $ 
(see definition~18 in \cite{GLT}), we may extend the equality $ V(y)  = V(x) + h(x,y) $. As in proposition~19 in \cite{GLT},
also for the summable variation case, each irreducible class is compact and invariant, so that it contains the support of at
least one minimizing measure.  

\emph{Item~\ref{corollary:RepresentationFormula:1}.} It suffices to explain how to show that $ h(x, \cdot) $, $ x \in \text{Mather}(H)$, is a calibrated sub-action.
The argument is standard. For $ x \in \text{Mather}(H)$, one may use Atkinson's theorem \cite{Atkinson} to obtain that, 
as a function of the second variable, $ h(x, \cdot) $ is finite everywhere on $ \Sigma $. Then the calibration property follows
from the very definition of the Peierls barrier. For details, see \cite{GLT, GaribaldiThieullen2012} and the references therein.
\end{proof}

\section{Explicit formulas for double-well type potentials}
\label{section:explicitFomulas}

From now on, we assume  that $H$ is a double-well type potential (see Definition~\ref{assumption:locallyConstant}). 
We show in lemma~\ref{lemma:CohomologousPotentials} that we can reduce the complexity of the notation by taking a suitable coboundary,
which is constant on a countable infinity of cylinders. As the issue of selection or nonselection is independent of the cohomological class of the potential, this lemma will enable us to simplify the proof by using the following reduced assumptions.

\begin{reducedAssumption}
Let $H$ be a double-well type potential. We say that $H$ is reduced if $H=0$ on $[00] \cup [11]$. More precisely, for every $n\geq0$,
\begin{enumerate}
\item  $H(x)=0$, \ if $x\in [00] \cup [11]$; 

\item$H(x)=H_n^0 > 0$, \  if $x\in [01^n0]$, \quad  $H(x)=H_n^1 > 0$, \  if $x\in[10^n1]$;

\item$\sum_{k\geq1} \sup_{n\geq0} |H_k^0 - H_{k+n}^0| < +\infty$, \  $\sum_{k\geq1} \sup_{n\geq0} |H_k^1 - H_{k+n}^1| < +\infty$.
\end{enumerate}
Denote
\begin{gather*}
H_\infty^0 := \lim_{n\to+\infty} H_n^0, \qquad H_\infty^1 := \lim_{n\to+\infty}H_n^1, \\
  H_{min}^0 := \inf_{n\geq1}H_n^0, \qquad H_{min}^1 := \inf_{n\geq1}H_n^1.
\end{gather*}
\end{reducedAssumption}

\begin {lemma} \label{lemma:CohomologousPotentials}
If $H$ is  double-well type potential, then there exists a function $V: \Sigma \to \mathbb{R}$, which is constant on a countable infinity of
cylinders, such that  $\tilde H := H -(V\circ\sigma-V)$ is reduced.
\end{lemma}

\begin{proof}
Let 
\begin{align*}
& V(x) := \sum_{k=n}^{+\infty} a_k^0  + \sum_{k\geq1} a_k^1, \qquad \text{if  $x \in [0^{n}1]$ and $n\geq1$}, \\
& V(x) := \sum_{k=n}^{+\infty} a_n^1 + \sum_{k\geq1} a_k^0,  \qquad \text{if $x \in [1^n0]$ and $n\geq1$}.
\end{align*}
Then
\begin{equation*}
V\circ\sigma-V =
\begin{cases}
\sum_{k\geq n} a_k^0 - \sum_{k\geq n+1} a_k^0 = a_n^0,                                         & \text{on $[00^n1]$}, \\ 
\sum_{k\geq n} a_k^1 - \sum_{k\geq n+1} a_k^1 = a_n^1,                                         & \text{on $[11^n0]$}, \\ 
(\sum_{k \geq n} a_k^0 + \sum_{k \geq 1} a_k^1) - (\sum_{k\geq1} a_k^1 + \sum_{k\geq1} a_k^0), & \text{on $[10^n1]$}, \\
(\sum_{k \geq n} a_k^1 + \sum_{k \geq 1} a_k^0) - (\sum_{k\geq1} a_k^0 + \sum_{k\geq1} a_k^1), & \text{on $[01^n0]$}.
\end{cases}
\end{equation*}
And the new double-well type potential $\tilde H := H -(V\circ\sigma-V)$ becomes
\begin{align*}
\tilde H(x) &= 0, &  \text{if $x\in [00] \cup [11]$}, \\
\tilde H(x) &= H_n^0 := b_n^0 +\sum_{k=1}^{n-1} a_k^1, & \text{if $x \in [01^n0]$}, \\
\tilde H(x) &= H_n^1 := b_n^1 +\sum_{k=1}^{n-1} a_k^0, & \text{if $x\in [10^n1]$}. & \qedhere
\end{align*}
\end{proof}

From now on, $H$ is supposed to be a reduced double-well type potential. We follow the same methods as in \cite{BaravieraLeplaideurLopes2012} and \cite{Leplaideur2012}. Our main goal is to find the characteristic equation of the eigenvalue $\lambda_\beta$ and the measures 
$\mu_\beta([0])$ and $\mu_\beta([1])$. We also want to identify the criterion of divergence in terms of the Peierls barrier.

Since $H$ is nonnegative and $H(0^\infty)=H(1^\infty)=0$, $H$ has null ergodic minimizing value:  $\bar H=0$. Since $\{0^\infty,1^\infty\}$ is the only invariant set included in $\{H=0\} \subset [00] \cup [11] \cup \{01^\infty,10^\infty\}$, the Mather set is reduced to the two fixed points, namely, $\text{\rm Mather}(H) = \{0^\infty,1^\infty\}$. 

The next proposition gives a complete description of the Peierls barrier.

\begin{proposition}
If $H$ is a reduced double-well type potential, then 
\begin{enumerate}
\item $h(0^\infty,x) = 0$, \ $\forall \, x \in [0]$, \hfill (in particular  $h(0^\infty,0^\infty)=0$); 

\item $h(0^\infty,x) = \inf_{k\geq n} H_k^0$, \   $\forall \, x \in [1^n0]$, \hfill (in particular  $h(0^\infty,1^\infty)= H_\infty^0$);

\item $ \liminf_{x \to 0^\infty} h(x,0^\infty) = H_{min}^0+H_\infty^1$;

\item $h(1^\infty,x) = 0$,\ $\forall \, x \in [1]$, \hfill  (in particular $h(1^\infty,1^\infty)=0$);

\item $h(1^\infty,x) = \inf_{k\geq n} H_k^1$, \ $\forall \, x \in [0^n1]$, \hfill (in particular $h(1^\infty,0^\infty)= H_\infty^1$);

\item $\liminf_{x \to 1^\infty} h(x,1^\infty) = H_{min}^1+H_\infty^0$.
\end{enumerate}
\end{proposition}

\begin{proof} \

{\it Item 1.} Clearly $h(0^\infty,x) = 0$, $\forall \, x \in [0]$, since $H \geq 0$ and  $H=0$  on $[00]$.  

{\it Item 2.} Let $x \in[1^n0]$ and $p\geq1$. Every  $z\in\Sigma$ satisfying $z \stackrel{p}{=} 0^\infty $ and 
$\sigma^k(z) \stackrel{p}{=} x$ belongs to $[0^{m_1}1^{n_1}\ldots 0^{m_r}1^{n_r}0]$,  with $m_1 \geq p$, $n_r \geq n$ and $k=m_1+n_1+\cdots+n_r-n$. 
 The corresponding sum $\sum_{i=0}^{k-1} [H \circ \sigma^i(z)-\bar H]$  is  $H_{n_1}^0+H_{m_2}^1 + \cdots + H_{n_r}^0$, 
 which gives (for every $m\geq p$)
\[
S_m^p(0^\infty,x) = \inf_{k \geq n}  H_k^0, \quad  h(0^\infty,x) = \inf_{k \geq n}  H_k^0.
\]
By continuity of $x \mapsto h(0^\infty,x)$ (see proposition~\ref{corollary:RepresentationFormula}), we have $h(0^\infty,1^\infty)=H_\infty^0$.

{\it Item 3.} On the one hand, if $x \in [0]$, $x \not= 0^\infty$ and $p \geq 1$, then every  $z$ satisfying $z \stackrel{p}{=} x $ and $\sigma^k(z)  \stackrel{p}{=} 0^\infty$ has the form $z=0^{m_1}1^{n_1}\cdots 0^{m_r}1^{n_r}0^p \cdots$ with $m_i\geq 1$, $n_i \geq 1$ and $k=m_1+n_1+\cdots+n_r$. The corresponding sum $\sum_{i=0}^{k-1} [H \circ \sigma^i(z)-\bar H]$  is bounded from below by  $ H_{min}^0+ \inf_{q\geq p}H_q^1$ and we obtain  $h(x,0^\infty) \geq H_{min}^0+ H_\infty^1$. On the other hand, for every $m, n\geq1$ and $k \geq  p \geq m+n$, $S_k^p(0^m1^n0^\infty,0^\infty)=H_n^0+H_\infty^1$. These facts together imply
\[
\liminf_{x\to 0^\infty} h(x,0^\infty)=H_{min}^0 + H_\infty^1.
\]

The other expressions are similarly obtained by permuting $0$ and $1$. 
\end{proof}

We recall  the notion of a Jacobian $J$  of a probability measure $\nu$ which is not necessarily invariant by the shift $\sigma$.  It is  a nonnegative Borel function $J : \Sigma \to \mathbb{R}^+$  such that, for every bounded Borel test function $f:\Sigma \to \mathbb{R}$,
\[
\int_{[0]} f \circ \sigma(x)J(x) \, d\nu(x) = \int_{[1]} f \circ \sigma(x)J(x) \, d\nu(x) = \int_\Sigma f(x) \, d\nu(x).
\]
Note that, if such a Jacobian  exists, it is unique. 

From now on, whenever a function $ f : \Sigma \to \mathbb R $ is constant on a cylinder $ [i_0 i_1 \ldots i_{n-1}] $, we denote 
$ f(i_0 i_1 \ldots i_{n-1}) $ the constant value $ f |_{[i_0 i_1 \ldots i_{n-1}]} $.

\begin{proposition} \label{proposition:BasicFormulas}
Let $H$ be a reduced double-well type potential. Let  $\Phi_\beta$, $\nu_\beta$ and $\lambda_\beta$ be the solutions of the Perron-Frobenius equation as defined in theorem~\ref{theorem:RuelleOperator}. Then  $\Phi_\beta$ is constant on every cylinder $[0^n1]$ or $[1^n0]$, $n\geq1$, and $\nu_\beta$ has constant Jacobian $J_\beta$ on the cylinders $[0^2]$, $[1^2]$, $[01^n0]$ and $[10^n1]$, $n\geq1$. More precisely,
\begin{enumerate}
\item \label{item:BasicFormulas_1} 
$\displaystyle \Phi_\beta(0^n1) = \sum_{k\geq n} \frac{\exp(-\beta H_k^1)}{\lambda_\beta^{k-n+1}}  \Phi_\beta (10) $, 
\,\, $\displaystyle \Phi_\beta(0^\infty) = \frac{\exp(-\beta H_\infty^1)}{\lambda_\beta -1} \Phi_\beta (10) $;
\item 
\label{item:BasicFormulas_2} $\displaystyle \Phi_\beta(1^n0) = \sum_{k\geq n} \frac{\exp(-\beta H_k^0)}{\lambda_\beta^{k-n+1}} \Phi_\beta(01)$, 
\,\, $\displaystyle \Phi_\beta(1^\infty) = \frac{\exp(-\beta H_\infty^0)}{\lambda_\beta -1} \Phi_\beta(01)$;
\item if $H_\infty^0= H_\infty^1=0$, then $\max \Phi_\beta = \max\{\Phi_\beta(0^\infty),\Phi_\beta(1^\infty)\}=1$;
\item \label{proposition:BasicFormulas:4}$\displaystyle \nu_\beta[1^n0] = \frac{1}{\lambda_\beta^{n-1}}\nu_\beta[10]$, or $J_\beta(x) = \lambda_\beta$, $\forall \, x\in [1^2]$;
\item $\displaystyle \nu_\beta[0^n1] = \frac{1}{\lambda_\beta^{n-1}}\nu_\beta[01]$, or $J_\beta(x) = \lambda_\beta$, $\forall \, x\in [0^2]$;
\item \label{proposition:BasicFormulas:6} $\displaystyle \nu_\beta[01^n0] = \frac{\exp(-\beta H_n^0)}{\lambda_\beta^n}  \nu_\beta[10]$, or $\displaystyle J_\beta(x) = \frac{\lambda_\beta}{\exp(-\beta H_n^0)}$, $\forall \, x \in [01^n0]$;
\item $\displaystyle \nu_\beta[10^n1] = \frac{\exp(-\beta H_n^1)}{\lambda_\beta^n}  \nu_\beta[01]$, 
or $\displaystyle J_\beta(x) = \frac{\lambda_\beta}{\exp(-\beta H_n^1)}$, $\forall \, x \in [10^n1]$.
\end{enumerate}
\end{proposition}

\begin{proof} \

{\it Part 1.} The equation $\mathcal{L}_\beta[\Phi_\beta]= \lambda_\beta \Phi_\beta$ implies
\begin{align*}
\Phi_\beta(0^n1) &= \frac{1}{\lambda_\beta} \Phi_\beta(0^{n+1}1) + \frac{1}{\lambda_\beta}\exp(-\beta H_n^1) \Phi_\beta(10) \\
&= \frac{1}{\lambda_\beta^2}\Phi_\beta(0^{n+2}1) +
 \Big[ \frac{1}{\lambda_\beta}\exp(-\beta H_n^1) + \frac{1}{\lambda_\beta^2}\exp(-\beta H_{n+1}^1)  \Big] \Phi_\beta(10) \\
&= \cdots = \Big[ \frac{1}{\lambda_\beta}\exp(-\beta H_n^1) + \frac{1}{\lambda_\beta^2}\exp(-\beta H_{n+1}^1) + \cdots  \Big] \Phi_\beta(10).
\end{align*}
A similar computation is done for $\Phi_\beta(1^n0)$.

{\it Part 2.} For every bounded Borel function $f:\Sigma\to\mathbb{R}$, we have
\begin{equation*}
\int\mathbb{1}_{[0]} f\circ \sigma \frac{\lambda_\beta}{\exp(-\beta H)} \, d\nu_\beta 
= \int \frac{\mathcal{L}_\beta}{\lambda_\beta} \Big[ \mathbb{1}_{[0]} f\circ \sigma  \frac{\lambda_\beta}{\exp(-\beta H)} \Big] \, d\nu_\beta = \int f \,d\nu_\beta.
\end{equation*}
A similar computation is done for $\mathbb{1}_{[1]}$. We thus obtain 
\begin{gather*}
J_\beta(x) = \frac{\lambda_\beta}{\exp(-\beta H(x))}, \quad \forall \, x \in \Sigma.
\end{gather*}
In particular, $J_\beta(x) =\lambda_\beta$ for $x \in [0^2] \cup [1^2]$,  $J_\beta(x) ={\lambda_\beta}/{\exp(-\beta H_n^0)} $ 
for $x \in [01^n0]$, and $J_\beta(x) ={\lambda_\beta}/{\exp(-\beta H_n^1)} $ for $x \in [10^n1]$.

{\it Part 3.} With respect to the eigenmeasure, we discuss items~\ref{proposition:BasicFormulas:4} and~\ref{proposition:BasicFormulas:6}; the others are similarly proved. Hence, by applying the Jacobian, just note that
\begin{align*}
\nu_\beta[10] &= \lambda_\beta \nu_\beta[1^20] = \lambda_\beta^2 \nu_\beta[1^30] = \cdots = \lambda_\beta^{n-1}\nu_\beta[1^n0] \\
&= \frac{\lambda_\beta^n}{\exp(-\beta H_n^0)} \nu_\beta[01^n0]. \qedhere
\end{align*}
\end{proof}

For every reduced double-well type potential, we define the following analytic functions that will play a fundamental role in the dichotomy:
\begin{gather}
 F_\beta^0(\lambda) := \sum_{k\geq 1} \frac{1}{\lambda^k} \exp(-\beta H_k^0), \qquad   F_\beta^1(\lambda) := \sum_{k\geq 1} \frac{1}{\lambda^k} \exp(-\beta H_k^1), \label{equation:F} \\
\tilde  F_\beta^0(\lambda) := \sum_{k\geq 1} \frac{k}{\lambda^k} \exp(-\beta H_k^0), \qquad   \tilde  F_\beta^1(\lambda) := \sum_{k\geq 1} \frac{k}{\lambda^k} \exp(-\beta H_k^1). \label{equation:Ftilde}
\end{gather}
We will also keep in mind the following equalities
\begin{equation} \label{equation:simpleSeries}
\forall \, N\geq0, \quad \sum_{k\geq N+1} \frac{1}{\lambda^k} = \frac{1}{\lambda^N(\lambda-1)}, \quad \sum_{k\geq N+1} \ \frac{k}{\lambda^k} = \frac{N(\lambda-1)+\lambda}{\lambda^N(\lambda-1)^2}.
\end{equation}

\begin{corollary} \label{corollary:GibbsMeasureEstimates}
Let $H$ be a reduced double-well type potential. Then
\begin{enumerate}
\item \label{corollary:GibbsMeasureEstimates_1} $F_\beta^0(\lambda_\beta)F_\beta^1(\lambda_\beta) = 1$ \qquad (the characteristic equation);
\item \label{corollary:GibbsMeasureEstimates_2} $\Phi_\beta(01) = F_\beta^1(\lambda_\beta) \Phi_\beta(10)$, \quad $\Phi_\beta(10) = F_\beta^0(\lambda_\beta) \Phi_\beta(01)$;
\item \label{corollary:GibbsMeasureEstimates_3} $\nu_\beta[01] = F_\beta^0(\lambda_\beta) \nu_\beta[10]$, \quad $\nu_\beta[10] = F_\beta^1(\lambda_\beta) \nu_\beta[01]$.
\end{enumerate}
\end{corollary}

\begin{proof}
Item~\ref{item:BasicFormulas_1} of proposition~\ref{proposition:BasicFormulas} implies, by taking $n=1$,
\begin{gather*}
\Phi_\beta(01) = F_\beta^1(\lambda_\beta) \Phi_\beta(10) \quad \text{and} \quad \Phi_\beta(10) = F_\beta^0(\lambda_\beta) \Phi_\beta(01).
\end{gather*}
By multiplying term to term, we obtain $F_\beta^0(\lambda_\beta)F_\beta^1(\lambda_\beta) = 1$. We also have
\begin{gather*}
\nu_\beta[01] = \sum_{n\geq1} \nu_\beta[01^n0]  = \sum_{n\geq1} \frac{1}{\lambda_\beta^n} \exp(-\beta H_n^0) \nu_\beta[10] = F_\beta^0(\lambda_\beta)\nu_\beta[10]. \qedhere
\end{gather*}
\end{proof}

\begin{corollary} \label{corollary:estimatesForGibbsMeasure}
Let $H$ be a reduced double-well type potential. Then
\begin{enumerate}
\item \label{corollary:estimatesForGibbsMeasure_1} $\mu_\beta[01] = \mu_\beta[10]$; 
\item \label{corollary:estimatesForGibbsMeasure_2} $\displaystyle \frac{\mu_\beta[0^n1]}{\mu_\beta[01]} = \Big[ \sum_{k\geq n}\frac{1}{\lambda_\beta^k} \exp(-\beta H_k^1)  \Big] F_\beta^0(\lambda_\beta)$, \quad $\displaystyle{\frac{\mu_\beta[0]}{\mu_\beta[01]} = \frac{\tilde F_\beta^1(\lambda_\beta)}{F_\beta^1(\lambda_\beta)} }$;
\item \label{corollary:estimatesForGibbsMeasure_3} $\displaystyle \frac{\mu_\beta[1^n0]}{\mu_\beta[10]} = \Big[ \sum_{k\geq n}\frac{1}{\lambda_\beta^k} \exp(-\beta H_k^0)  \Big] F_\beta^1(\lambda_\beta)$, \quad $\displaystyle{\frac{\mu_\beta[1]}{\mu_\beta[10]} = \frac{\tilde F_\beta^0(\lambda_\beta)}{F_\beta^0(\lambda_\beta)} }$;
\item \label{corollary:estimatesForGibbsMeasure_4} $\displaystyle \frac{\mu_\beta[01^n0]}{\mu_\beta[10]} = \frac{\exp(-\beta H_n^0) F_\beta^1(\lambda_\beta)}{\lambda_\beta^n}$, \quad $\displaystyle \frac{\mu_\beta[10^n1]}{\mu_\beta[01]} = \frac{\exp(-\beta H_n^1) F_\beta^0(\lambda_\beta)}{\lambda_\beta^n}$;
\item \label{corollary:estimatesForGibbsMeasure_5}  $\displaystyle{\frac{\mu_\beta[0]}{\mu_\beta[1]} = \frac{ F_\beta^0(\lambda_\beta)}{F_\beta^1(\lambda_\beta)} \frac{\tilde F_\beta^1(\lambda_\beta)}{\tilde F_\beta^0(\lambda_\beta)}}$.
\end{enumerate}
\end{corollary}

We know that $\lambda_\beta \to 1$ as $\beta \to +\infty$. In order to understand the behavior of $\mu_\beta$, it is fundamental to have a better Puiseux series expansion of $\lambda_\beta$, as it is done for potentials that depend on finite number of coordinates (see \cite{GaribaldiThieullen2012}). The log-scale limit, the limit of $-\frac{1}{\beta} \ln(\lambda_\beta-1)$,  is usually easy to obtain using a min-plus technique. This may be sufficient to show the convergence of $\mu_\beta$ when there is no coincidence of exponents, as it happens in \cite{BaravieraLopesMengue2013_01}. Usually the limit is then a periodic measure.  In general, the log-scale limit is not sufficient and an expansion of the form $\lambda_\beta = 1 + c e^{-\beta \gamma} + o(e^{-\beta\gamma})$ needs to be founded as in \cite{BaravieraLeplaideurLopes2012,Leplaideur2012}. A barycenter of periodic measures with irrational coefficients may be the limit in this case.
Let us recall from equation \eqref{equation:threeExponents} the definition of the key parameter $ \gamma $, which we call from now on the Puiseux exponent:
\[
\gamma := \min \Big\{\frac{1}{2} \big( H_\infty^1 + H_\infty^0 \big), \,  H_{min}^0 + H_\infty^1, \, H_{min}^1 + H_\infty^0\Big\}.
\]
The coincidence of exponents is understood in the sense that the minimum $\gamma$ may be attained several times. The following proposition gives the log-scale limit of the main quantities that appear in the dichotomy. We will give better estimates in the next section.

\begin{proposition} \label{proposition:second_exponent}
Let $H$ be a reduced double-well type potential. Then
\begin{enumerate} 
\item \label{item:1_second_exponent} $\displaystyle\lim_{\beta\to+\infty} -\frac{1}{\beta} \ln(\lambda_\beta - 1) = \gamma$; 
\item \label{item:2_second_exponent} $\displaystyle \lim_{\beta \to +\infty}-\frac{1}{\beta}\ln F_\beta^{0,1}(\lambda_\beta) 
= \min_{n\geq1}\big\{H_n^{0,1},\, H_\infty^{0,1}-\gamma \big\}$;
\item \label{item:3_second_exponent} $\displaystyle \lim_{\beta \to +\infty} -\frac{1}{\beta} \ln \tilde F_\beta^{0,1}(\lambda_\beta)
= \min_{n\geq1} \big\{H_n^{0,1},\, H_\infty^{0,1} - 2\gamma \big\}$.
\end{enumerate}
\end{proposition}

\begin{proof} \

{\it Part 1.} We claim that any limit point of $-\frac{1}{\beta}\ln(\lambda_\beta -1)$ is finite. 
Recall that $H$ is nonnegative and $ \max \Phi_\beta = 1 $.
Hence, given $ x_\beta^\text{max} \in \argmax \Phi_\beta $, we see that 
$ \lambda_\beta = \mathcal L_\beta[\Phi_\beta](x_\beta^\text{max}) \le 2 $. Since 
$ \lambda_\beta \Phi_\beta(0^\infty) = \mathcal L_\beta[\Phi_\beta](0^\infty) $ yields 
$ \lambda_\beta = 1 + \exp(-\beta H_\infty^1) \Phi_\beta(10^\infty) / \Phi_\beta(0^\infty) \ge 1 $,  
we have the \emph{a priori} estimate $1\leq \lambda_\beta \leq 2$. Furthermore, from
\[
\frac{\exp(-\beta \max_k H_k^0)}{\lambda_\beta - 1} \le F_\beta^0(\lambda_\beta) = \frac{1}{F_\beta^1(\lambda_\beta)} \le \frac{\lambda_\beta - 1}{\exp(-\beta \max_k H_k^1)},
\]
we conclude that $ \exp\big(\!-\beta(\max H_k^0 + \max H_k^1)/2 \big) \le \lambda_\beta - 1 \le 1 $.

{\it Part 2.} For some subsequence $\beta \to +\infty$, assume $-\frac{1}{\beta}\ln(\lambda_\beta -1) \to \bar\gamma$. 
We claim that $-\frac{1}{\beta}\ln F_\beta^0(\lambda_\beta) \to \min_{n\geq1}(H_n^0,H_\infty^0-\bar\gamma)$ for the same subsequence. 
Indeed, let $\epsilon>0$. We choose $N\geq1$ such that $|H_n^0-H_\infty^0| <  \epsilon$  for all $n\geq N$. We split the series \eqref{equation:F} in two terms. For the first term, for $\beta$ large enough
\[
\exp(-\beta(\min_{1 \leq k \leq N} H_k^0+\epsilon)) \leq \sum_{k=1}^N \frac{1}{\lambda_\beta^k} \exp(-\beta H_k^0)  \leq \exp(-\beta(\min_{1 \leq k \leq N} H_k^0-\epsilon)).
\]
For the second term,  using the estimates \eqref{equation:simpleSeries}, for $ \beta $ large enough
\begin{gather*}
\exp(-\beta(\bar \gamma+ \epsilon)) \leq \lambda_\beta^N(\lambda_\beta -1) \leq \exp(-\beta (\bar\gamma-\epsilon)), \\
\frac{\exp(-\beta(H_\infty^0+\epsilon))}{\lambda_\beta^N(\lambda_\beta-1)} \leq \sum_{k > N} \frac{1}{\lambda_\beta^k} \exp(-\beta H_k^0) \leq  \frac{\exp(-\beta(H_\infty^0-\epsilon))}{\lambda_\beta^N(\lambda_\beta-1)}, \\
\exp(-\beta(H_\infty^0-\bar\gamma+2\epsilon)) \leq \sum_{k > N} \frac{1}{\lambda_\beta^k} \exp(-\beta H_k^0)  
\leq \exp(-\beta(H_\infty^0-\bar\gamma-2\epsilon)).
\end{gather*}
The claim is proved  by adding the two terms, changing the scale and passing to the limits as $\beta  \to +\infty$ and $\epsilon \to 0$.

{\it Part 3.} We show there is a unique limit point $\bar\gamma$ by showing that it is the unique solution of a min-plus equation. 
Indeed, from the characteristic equation $1=F_\beta^0(\lambda_\beta) F_\beta^1(\lambda_\beta)$, we obtain
\[
0 = \min_{n\geq1}\{H_n^0,H_\infty^0-\bar\gamma\} + \min_{n\geq1}\{H_n^1,H_\infty^1-\bar\gamma\}.
\]
This equation is equivalent to
\begin{equation*}
\min_{n\geq1} H_n^0 + H_\infty^1-\bar \gamma = 0 \ \text{ or } \ \min_{n\geq1} H_n^1 + H_\infty^0-\bar\gamma = 0 \ \text{ or } \ 
H_\infty^0+ H_\infty^1-2\bar\gamma=0.
\end{equation*}
We have shown that $\bar \gamma$ is the Puiseux exponent $ \gamma $.

{\it Part 4.} We prove  item \ref{item:3_second_exponent} similarly as in part 2. We choose $\epsilon>0$ and $N \ge 1$ as before. The first part of the series \eqref{equation:Ftilde} satisfies
\[
\lim_{\beta \to +\infty} -\frac{1}{\beta} \ln \sum_{k=1}^N
 \frac{k}{\lambda_\beta^k} \exp(-\beta H_k^0) = \min_{1 \leq k \leq N} H_k^0.
\]
 Using again the estimate \eqref{equation:simpleSeries}, for $ \beta $ large enough, the remaining part gives
\begin{gather*}
\exp(-\beta(2\gamma+\epsilon)) \leq \frac{\lambda_\beta^N(\lambda_\beta-1)^2}{N(\lambda_\beta-1)+\lambda_\beta} 
\leq \exp(-\beta(2\gamma-\epsilon)), \\
\exp(-\beta(H_\infty^0-2\gamma+2\epsilon)) \leq \sum_{k > N} \frac{k}{\lambda_\beta^k} \exp(-\beta H_k^0) 
\leq \exp(-\beta(H_\infty^0-2\gamma-2\epsilon)). \qedhere
\end{gather*}
\end{proof}

\begin{corollary} \label{proposition:representationFormula}
Let  $H$ be a reduced double-well type potential and $V$ be a calibrated sub-action. Then  $V$  is constant on every cylinders of the form $[0^n1]$ and $[1^n0]$ where $n\geq1$. More precisely,
\begin{enumerate}
\item \label{item:representationFormula_1} $\displaystyle V(x) =  \min \big\{ V(0^\infty), V(1^\infty)+ \inf_{k \geq n} H_k^1 \big\}$, \ $\forall \, x \in [0^n1]$, 
\item \label{item:representationFormula_2} $\displaystyle V(x) =\min \big\{ V(1^\infty), V(0^\infty)+  \inf_{k \geq n}  H_k^0 \big\}$, \ $\forall \, x \in [1^n0]$.
\end{enumerate}
In particular, $\min V = \min\{V(0^\infty),V(1^\infty)\}$. With respect to $\Phi_\beta=e^{-\beta V_\beta}$ the eigenfunction used 
in theorem~\ref{theorem:RuelleOperator} to ensure the existence of calibrated sub-actions, we have the following complementary information. 
\begin{enumerate}
\setcounter{enumi}{2}
\item \label{item:representationFormula_3} If $\gamma>0$ and $H_\infty^1 \geq H_\infty^0$, then $\{V_\beta\}$ converges uniformly to the calibrated sub-action $V_\infty$ characterized by
\begin{align*}
&V_\infty(x) = \min\{H_\infty^1 - \gamma, \inf_{k\geq n}H_k^1\}, \quad \forall \, x \in [0^n1], \ \forall \, n\geq1, \\ 
&V_\infty(x)  = 0, \quad \forall \, x \in [1].
\end{align*}
\item \label{item:representationFormula_4}  If  $\gamma=0$, then  $\{V_\beta\}$ converges uniformly to  $ 0$, which  is the unique calibrated sub-action satisfying $\min V =0$.
\end{enumerate}

\end{corollary}

\begin{proof} \

{\it Part 1.} Items~\ref{item:representationFormula_1} to~\ref{item:representationFormula_2} are consequences of the representation 
formula~\eqref{formula representacao}.

{\it Part 2.} If $H_\infty^1 \geq H_\infty^0$, then $H_\infty^1 + H_\infty^0 - 2\gamma \geq 0 \ge H_\infty^0 - \gamma $. 
Item~\ref{item:BasicFormulas_1} of proposition~\ref{proposition:BasicFormulas}, item~\ref{corollary:GibbsMeasureEstimates_2} 
of corollary~\ref{corollary:GibbsMeasureEstimates} and items~\ref{item:1_second_exponent} and~\ref{item:2_second_exponent} 
of proposition~\ref{proposition:second_exponent} imply
\[
\lim_{\beta \to +\infty} \big[V_\beta(0^\infty) - V_\beta(01)\big] = H_\infty^1 + H_\infty^0 - 2\gamma \geq 0.
\]
From~item~\ref{item:BasicFormulas_2} of proposition~\ref{proposition:BasicFormulas} and item~\ref{item:1_second_exponent} 
of proposition~\ref{proposition:second_exponent}, we have
\[
\lim_{\beta \to +\infty} \big[V_\beta(1^\infty) - V_\beta(01)\big] = H_\infty^0 - \gamma \le 0.
\]
Therefore, we obtain
\[
\lim_{\beta \to +\infty} \big[V_\beta(0^\infty) - V_\beta(1^\infty)\big] = H_\infty^1 - \gamma \geq 0.
\]
Let $V_\infty$ be any accumulation function of $\{V_\beta\}$. Then $V_\infty$ is a calibrated sub-action and, in particular, 
satisfies items~\ref{item:representationFormula_1} and~\ref{item:representationFormula_2} already proved. 
Thus, since $\min V_\infty =0$, necessarily $V_\infty(1^\infty)=0$ and $V_\infty(0^\infty)=H_\infty^1-\gamma$, so that the characterization
given in item~\ref{item:representationFormula_3} is proved. Being the limit function uniquely defined, we have actually showed that 
$V_\beta \to V_\infty$ uniformly.

{\it Part 3.} If $\gamma=0$, then  $H_\infty^0= H_\infty^1=0$. Let $V_\infty$ be any accumulation function of $\{V_\beta\}$. 
Then $V_\infty$ is a calibrated sub-action. By passing to the limit as $n\to+\infty$ in items \ref{item:representationFormula_1} and \ref{item:representationFormula_2}, we obtain $V_\infty(0^\infty)=V_\infty(1^\infty)$. Since $\min V_\infty =0$, $V_\infty$ is necessarily the null function. By uniqueness of the accumulation function, we have proved that $V_\beta \to V_\infty$ uniformly.
\end{proof}

\section{The selection case}
\label{section:selectionCase}

We assume  that $H$ is reduced and that $\gamma>0 $, which is equivalent to $ \max \{H_\infty^0,H_\infty^1\}>0$.
We also suppose that $H_\infty^0 \leq H_\infty^1$  (the opposite case is similar). 
In particular, $H_\infty^1>0$. We know that the  only accumulation points of $\mu_\beta$ are barycenters $c_0\delta_{0^\infty}+c_1\delta_{1^\infty}$. Our goal is to find an equivalent of $\mu_\beta[0]/\mu_\beta[1]$  as $\beta\to+\infty$ and therefore to prove the convergence of $\mu_\beta$.

\begin{proof}[Proof of item \ref{item:mainTheorem_0} of Theorem \ref{theorem:main}]
Assume $ \frac{1}{2}\big( H_\infty^1 + H_\infty^0 \big) > H_{min}^1+H_\infty^0$. Then $\gamma = H_{min}^1+H_\infty^0 > 0$ since $H_{min}^1=0 \ \Leftrightarrow \ H_\infty^1=0$. We will see that it is enough to estimate the quotient of the measures at the log-scale. Proposition~\ref{proposition:second_exponent} implies
\begin{align*}
&\lim_{\beta \to +\infty} -\frac{1}{\beta}\ln F_\beta^0(\lambda_\beta) = \min\{H_{min}^0,H_\infty^0-\gamma\} =H_\infty^0-\gamma, \\
&\lim_{\beta \to +\infty} -\frac{1}{\beta}\ln \tilde F_\beta^0(\lambda_\beta) =  \min\{H_{min}^0,H_\infty^0-2\gamma\} = H_\infty^0-2\gamma, \\
&\lim_{\beta \to +\infty} -\frac{1}{\beta}\ln \tilde F_\beta^1(\lambda_\beta) =  \min\{H_{min}^1,H_\infty^1-2\gamma\}.
\end{align*}
The estimate for $F_\beta^1$ is obtained from the characteristic equation. Thus
\begin{align*}
\lim_{\beta \to +\infty} -\frac{1}{\beta} \ln \Big( \frac{\mu_\beta[0]}{\mu_\beta[1]} \Big) &= \lim_{\beta \to +\infty} -\frac{1}{\beta} \ln \Big( \frac{ F_\beta^0(\lambda_\beta)}{F_\beta^1(\lambda_\beta)} \frac{\tilde F_\beta^1(\lambda_\beta)}{\tilde F_\beta^0(\lambda_\beta)} \Big), \\
&= H_\infty^0 + \min\{H_{min}^1,H_\infty^1-2\gamma\} > 0.
\end{align*}
We have proved that $\mu_\beta[0]/\mu_\beta[1] \to 0$ or $\mu_\beta \to \delta_{1^\infty}$.
\end{proof}

For the proof of  item~\ref{item:mainTheorem_1} of theorem~\ref{theorem:main}, the previous log-scale estimate is not enough. 
We need to develop an analytical  technique which gives  equivalents of the quantities $F_\beta^{0,1}(\lambda_\beta)$, $\tilde F_\beta^{0,1}(\lambda_\beta)$, and  $\lambda_\beta-1$.

We first need the following lemma on sequences.

\begin{lemma} \label{lemma:SimpleTauberian}
Let $\{H_n\}_{n\geq0}$ be a converging sequence satisfying 
\[
\sum_{n\geq0} \ \sup_{k\geq0} \ | H_n - H_{n+k} | < +\infty.
\]
Then $\lim_{n\to+\infty} (H_n-H_\infty) \ln(n) = 0$, where $H_\infty = \lim_{n\to+\infty} H_n$.
\end{lemma}

\begin{proof}
Denote $K_n :=  \sup_{k\geq0} \ | H_n - H_{n+k} |$ for all $ n \ge 0 $. Note then that $|H_n-H_\infty| \leq K_n$ and $\{K_n\}_{n\geq0}$ is a nonincreasing sequence converging to~$ 0 $ such that $\sum_{n\geq 0} K_n < +\infty$.
Assume by contradiction that there exist $\epsilon>0$ and a subsequence $N_i\to+\infty$ such that $K_{N_i}\ln(N_i) \geq  \epsilon$. Thanks to the nonincreasing property, we have
\[
\sum_{i\geq1} \frac{N_{i+1}-N_i}{\ln(N_{i+1})} \leq \frac{1}{\epsilon} \sum_{i\geq1} \sum_{N_i \leq n < N_{i+1}} K_{n} < + \infty.
\]
We thus observe that
\[
 \frac{1-N_{i}/N_{i+1}}{\ln(N_{i+1})/N_{i+1}} \to  0 \quad \Longrightarrow \quad \frac{N_i}{N_{i+1}}  \to 1,
\]
which implies, for every $i$ sufficiently large, 
\[
\frac{N_{i+1}-N_i}{\ln(N_{i+1})} = \frac{N_i}{\ln(N_{i+1})} \Big(\frac{N_{i+1}}{N_i}-1\Big)
\geq \frac{N_{i+1}}{N_i}-1 \geq \ln\Big(\frac{N_{i+1}}{N_i} \Big).
\]
But then $\sum_{i\geq1} [\ln(N_{i+1}) - \ln(N_i)] < +\infty$ contradicts $N_i \to +\infty$.
\end{proof}

From now on, we write $ f(\beta) \sim g(\beta) $ to indicate that the positive functions $ f $ and $ g $ are equivalent as 
$ \beta \to + \infty $. Besides, as usual $ f(\beta) \ll g(\beta) $ means that $ f $ is negligible with respect to $ g $ as 
$ \beta \to + \infty $.

\begin{proof}[Proof of item \ref{item:mainTheorem_1} of theorem \ref{theorem:main}]
Assume $0 < \frac{1}{2}\big( H_\infty^1 + H_\infty^0 \big) \leq H_{min}^1+H_\infty^0$. Then $\gamma =  \frac{1}{2}\big( H_\infty^0 + H_\infty^1 \big)$. We recall that the coincidence number $\kappa$ has been defined in \eqref{equation:coincidenceNumber} and the coefficient $c$ 
in \eqref{equation:PuiseuxCoefficient}. We will prove the following results:
\begin{align}
\begin{split}
&\lambda_\beta = 1 + c \exp(-\beta \gamma) + o(\exp(-\beta\gamma)),\\
&F_\beta^0(\lambda_\beta) \sim \frac{\exp(-\beta H_\infty^0)}{\lambda_\beta-1} \sim \frac{1}{c} \exp\Big( \beta \frac{H_\infty^1-H_\infty^0}{2} \Big), \\
&\tilde F_\beta^0(\lambda_\beta) \sim \frac{\exp(-\beta H_\infty^0)}{(\lambda_\beta-1)^2} \sim \frac{1}{c^2} \exp( \beta H_\infty^1), \\
&F_\beta^1(\lambda_\beta) \sim c \exp\Big( -\beta \frac{H_\infty^1-H_\infty^0}{2} \Big), \\
&\tilde F_\beta^1(\lambda_\beta) \sim \frac{\exp(-\beta H_\infty^1)}{(\lambda_\beta-1)^2} \sim \frac{1}{c^2} \exp( \beta H_\infty^0).
\end{split} \label{equation:selectionCase0}
\end{align}
Using item~\ref{corollary:estimatesForGibbsMeasure_5} of corollary~\ref{corollary:estimatesForGibbsMeasure}, 
we will obtain $\mu_\beta[0] / \mu_\beta[1] \to 1/c^2$  and  the convergence of the Gibbs measure as in~\eqref{equation:SelectionCase}. 

\medskip
{\it Part 1.} We determine an equivalent of $F_\beta^0(\lambda_\beta)$. If $H_k^0$ is constant and equal to $H_\infty^0$, we are done:
 \[
F_\beta^0(\lambda_\beta)= \frac{\exp(-\beta H_\infty^0)}{\lambda_\beta-1} \quad  \text{and} \quad  \tilde F_\beta^0(\lambda_\beta) = \frac{\exp(- \beta H_\infty^0)}{(\lambda_\beta-1)^2}.
\]
We may now assume that $H_k^0$ is not constant. Let $\epsilon>0$. For $\beta$ large enough, there exists a smallest positive integer $N_\beta$ 
such that 
\[
\beta |H_{N_\beta}^0 - H_\infty^0| \geq \epsilon, \quad\text{and}\quad \beta |H_k^0-H_\infty^0| \leq \epsilon, \ \forall \, k \geq N_\beta+1.
\]
Lemma~\ref{lemma:SimpleTauberian} implies that $|H_n^0-H_\infty^0| \ln (n) \to 0$. Since $|H_{N_\beta}^0-H_\infty^0| \geq \epsilon/\beta$, we obtain (even in the case $N_\beta$ is bounded with respect to $\beta$)
\begin{equation} \label{equation:selectionCase01}
\lim_{\beta \to +\infty} \frac{1}{\beta} \ln N_\beta = 0.
\end{equation}
Hence, we may show that
\begin{equation}  \label{equation:selectionCase02}
N_\beta(\lambda_\beta-1) \exp(-\beta H_{min}^0) \ll \exp(-\beta H_\infty^0) \quad \text{ and } \quad \lambda_\beta^{N_\beta} \to 1.
\end{equation}
For the first estimate,  by taking $-\frac{1}{\beta} \ln$ on both terms and using item~\ref{item:1_second_exponent} of 
proposition~\ref{proposition:second_exponent}, one has $\gamma + H_{min}^0 > H_\infty^0$ (according to the two cases: 
if  $H_\infty^1>H_\infty^0$ then $\gamma > H_\infty^0$, if $H_\infty^1=H_\infty^0$ then $H_{min}^0>0$). For the above limit, note that
\begin{gather*}
\frac{\lambda_\beta-1}{\exp(-\beta H_{min}^1)} \leq \frac{1}{F_\beta^1(\lambda_\beta)} = F_\beta^0(\lambda_\beta) 
\leq \frac{1}{\lambda_\beta-1}, \\
\lambda_\beta \leq 1+\exp(-\beta H_{min}^1/2), \quad \lambda_\beta^{N_\beta} \leq \exp\big(N_\beta\exp(-\beta H_{min}^1/2)\big).
\end{gather*}
As $H_{min}^1>0$, using \eqref{equation:selectionCase01} one gets  $N_\beta \ll \exp(\beta H_{min}^1/2)$ and $\lambda_\beta^{N_\beta} \to 1$. 

We are now able to compute an equivalent of $F_\beta^0(\lambda_\beta)$. We split the series $F_\beta^0(\lambda_\beta)$ in two parts and use~\eqref{equation:selectionCase02} to obtain, for $\beta$ sufficiently large,
\begin{gather*}
\frac{\exp(-\beta H_\infty^0-\epsilon)}{\lambda_\beta^{N_\beta}(\lambda_\beta-1)} \leq F_\beta^0(\lambda_\beta) \leq N_\beta \exp(-\beta H_{min}^0) + \frac{\exp(-\beta H_\infty^0+\epsilon)}{\lambda_\beta^{N_\beta}(\lambda_\beta-1)}, \\
\frac{\exp(-\beta H_\infty^0-2\epsilon)}{\lambda_\beta-1} \leq F_\beta^0(\lambda_\beta) \leq \frac{\exp(-\beta H_\infty^0+2\epsilon)}{\lambda_\beta-1}.
\end{gather*}
By taking $\epsilon \to 0$, we have just proved
\begin{equation} \label{equation:selectionCase03}
F_\beta^0(\lambda_\beta) \sim \frac{\exp(-\beta H_\infty^0)}{\lambda_\beta-1}.
\end{equation}

{\it Part 2.} We determine an equivalent of $\tilde F_\beta^0(\lambda_\beta)$. We use the same definition of $N_\beta$ as before and prove similarly the estimates
\begin{equation} \label{equation:selectionCase04}
N_\beta (\lambda_\beta-1) \ll 1, \quad N_\beta^2 (\lambda_\beta-1) ^2 \exp(-\beta H_{min}^0) \ll \exp(-\beta H_\infty^0).
\end{equation}
We split the series $\tilde F_\beta^0(\lambda_\beta)$ and use the computation \eqref{equation:simpleSeries} to obtain
\begin{multline*}
\frac{(N_\beta(\lambda_\beta-1)+\lambda_\beta)\exp(-\beta H_\infty^0-\epsilon)}{\lambda_\beta^{N_\beta}(\lambda_\beta-1)^2} 
\leq \tilde F_\beta^0(\lambda_\beta)  \\ 
\tilde F_\beta^0(\lambda_\beta)  \leq N_\beta^2 \exp(-\beta H_{min}^0) + \frac{(N_\beta(\lambda_\beta-1)+\lambda_\beta)\exp(-\beta H_\infty^0+\epsilon)}{\lambda_\beta^{N_\beta}(\lambda_\beta-1)^2}.
\end{multline*}
Using the estimates~\eqref{equation:selectionCase04}, one gets for $\beta$ sufficiently large
\[
\frac{\exp(-\beta H_\infty^0-2\epsilon)}{(\lambda_\beta-1)^2} \leq \tilde F_\beta^0(\lambda_\beta) \leq \frac{\exp(-\beta H_\infty^0+2\epsilon)}{(\lambda_\beta-1)^2}.
\]
Letting $\epsilon \to 0$, we have just proved
\begin{equation} \label{equation:selectionCase05}
\tilde F_\beta^0(\lambda_\beta) \sim \frac{\exp(-\beta H_\infty^0)}{(\lambda_\beta-1)^2}.
\end{equation}

{\it Part 3.} We determine an equivalent of $F_\beta^1(\lambda_\beta)$. As before we discuss two cases. If $H_k^1$ is constant and 
equal to $H_\infty^1$, the coincidence number~\eqref{equation:coincidenceNumber} is $ \kappa=0$ and the coefficient~\eqref{equation:PuiseuxCoefficient} is $c=1$. We immediately obtain 
\[
F_\beta^1(\lambda_\beta)= \frac{\exp(-\beta H_\infty^1)}{\lambda_\beta-1} \quad  \text{and} \quad  \tilde F_\beta^1(\lambda_\beta) = \frac{\exp(- \beta H_\infty^1)}{(\lambda_\beta-1)^2}.
\]
We may assume $H_k^1$ is not constant. For $ \beta $ large enough, we redefine $N_\beta$ as the smallest positive integer such that 
\[
\beta |H_{N_\beta}^1 - H_\infty^1| \geq \epsilon, \quad\text{and}\quad \beta |H_k^1-H_\infty^1| \leq \epsilon, \ \forall \, k \geq N_\beta+1.
\]
As before $\frac{1}{\beta} \ln N_\beta \ll 1$. Recall now that $ H_{min}^1 \ge \frac{1}{2}(H_\infty^1 - H_\infty^0) $. 
In the case $\kappa>0$, $H_{min}^1 < H_\infty^1$ and we introduce another exponent 
\[
H_{min}^{1*}:= \min \Big\{ H_k^1 \,:\, k \ \text{s.t.} \ H_k^1+H_\infty^0 \not= \frac{1}{2}\big( H_\infty^1+H_\infty^0 \big) \Big\} > H_{min}^1.
\]
In the case $\kappa=0$, by convention, $H_{min}^{1*} = H_{min}^1$. We show the first estimate
\begin{equation}  \label{equation:selectionCase06}
N_\beta (\lambda_\beta-1) \exp(-\beta H_{min}^{1*}) \ll \exp(-\beta H_\infty^1).
\end{equation}
Indeed, by taking $-\frac{1}{\beta} \ln$, it is enough to argue that $\gamma+H_{min}^{1*} > H_\infty^1$.  
In the case $\kappa >0$, $H_{min}^1 + H_\infty^0 = \frac{1}{2}(H_\infty^1+H_\infty^0) = \gamma$ and 
\[
\gamma  + H_{min}^{1*} >  \gamma  + H_{min}^1 = H_\infty^1.
\]
In the case $\kappa=0$, $H_{min}^1 + H_\infty^0 > \frac{1}{2}(H_\infty^1+H_\infty^0) = \gamma$ and
\[
\gamma + H_{min}^{1*} = \gamma + H_{min}^1 > H_\infty^1.
\]
The limit $\lambda_\beta^{N_\beta} \to 1$ is similarly proved. We are now able to compute an equivalent of $F_\beta^1(\lambda_\beta)$. We split as before the series in two parts: in the finite sum, we keep the indices corresponding to the incidences and the exponents $H_{min}^1$, the rest of the indices have an larger exponent $H_{min}^{1*}$ (unless $\kappa=0$ where we only use one exponent $H_{min}^1$). For $ \beta $ large
enough, we thus have
\begin{multline*}
(e^{-\epsilon} \kappa) \exp(-\beta H_{min}^1) + \frac{\exp(-\beta H_\infty^1 - \epsilon)}{\lambda_\beta^{N_\beta}(\lambda_\beta-1)} 
\leq F_\beta^1(\lambda_\beta) \\
F_\beta^1(\lambda_\beta) \leq \kappa \exp(-\beta H_{min}^1) + N_\beta \exp(-\beta H_{min}^{1*}) + \frac{\exp(-\beta H_\infty^1 + \epsilon)}{\lambda_\beta^{N_\beta}(\lambda_\beta-1)}.
\end{multline*}
Taking into account the estimate~\eqref{equation:selectionCase06}, for $\beta$ sufficiently large 
\begin{multline*}
\Big[ \kappa \exp(-\beta H_{min}^1) + \frac{\exp(-\beta H_\infty^1)}{\lambda_\beta-1} \Big] e^{-2\epsilon} 
\leq F_\beta^1(\lambda_\beta) \\
F_\beta^1(\lambda_\beta) \leq \Big[ \kappa \exp(-\beta H_{min}^1) +  \frac{\exp(-\beta H_\infty^1)}{\lambda_\beta-1} \Big] e^{2\epsilon},
\end{multline*}
Letting $\epsilon \to 0$, we have proved (in both cases, $\kappa>0$ or $\kappa=0$)
\begin{equation} \label{equation:selectionCase07}
F_\beta^1(\lambda_\beta) \sim  \kappa \exp(-\beta H_{min}^1) + \frac{\exp(-\beta H_\infty^1)}{\lambda_\beta-1}.
\end{equation}

{\it Part 4.} We show an equivalent of $\lambda_\beta-1$. The characteristic equation (item \ref{corollary:GibbsMeasureEstimates_1} of corollary \ref{corollary:GibbsMeasureEstimates}), the equivalents \eqref{equation:selectionCase03} and \eqref{equation:selectionCase07} give
\[
(\lambda_\beta -1)^2 \exp(\beta(H_\infty^1+H_\infty^0)) \sim  \kappa \, (\lambda_\beta-1) \exp(\beta(H_\infty^1+H_\infty^0)/2) + 1.
\]
(In the case $\kappa>0$, we use the equality $H_{min}^1+H_\infty^0 = \frac{1}{2} \big(H_\infty^1+H_\infty^0)$.) Let
$X_\beta = (\lambda_\beta-1) \exp(\beta(H_\infty^1+H_\infty^0)/2)$. Then $X_\beta^2 \sim \kappa X_\beta +1$.  Necessarily $X_\beta$ is bounded with respect to $\beta$, nonnegative, and any accumulation point $c$ satisfies $c^2 = \kappa c+1$. We have just proved that
\begin{equation} \label{equation:selectionCase08}
\lambda_\beta -1 \sim c \exp \Big(-\beta \frac{1}{2}(H_\infty^1+H_\infty^0) \Big).
\end{equation}
Using the previous equivalents~\eqref{equation:selectionCase03} and~\eqref{equation:selectionCase05} as well as the characteristic equation, 
one obtains the equivalents of $F_\beta^0(\lambda_\beta)$,  $\tilde F_\beta^0(\lambda_\beta)$ and $F_\beta^1(\lambda_\beta)$. For the equivalent of $\tilde F_\beta^1(\lambda_\beta)$, since $ 2 \gamma + H_{min}^1 = H_\infty^1 + H_{min}^1 + H_\infty^0 > H_\infty^1$,
one first notices that
\begin{equation} \label{equation:selectionCase09}
N_\beta^2 (\lambda_\beta-1)^2 \exp(-\beta H_{min}^1) \ll \exp(-\beta H_\infty^1).
\end{equation}
The series $\tilde F_\beta^1(\lambda_\beta)$ is then split in a more crude way
\begin{multline*}
\frac{(N_\beta(\lambda_\beta-1)+\lambda_\beta)\exp(-\beta H_\infty^1 - \epsilon)}{\lambda_\beta^{N_\beta}(\lambda_\beta-1)^2} 
\leq \tilde F_\beta^1(\lambda_\beta) \\
\tilde F_\beta^1(\lambda_\beta)  \leq  N^2_\beta \exp(-\beta H_{min}^{1}) + \frac{(N_\beta(\lambda_\beta-1)+\lambda_\beta)\exp(-\beta H_\infty^1 + \epsilon)}{\lambda_\beta^{N_\beta}(\lambda_\beta-1)^2},
\end{multline*}
and therefore
\begin{equation}
\tilde F_\beta^1(\lambda_\beta) \sim \frac{\exp(-\beta H_\infty^1)}{(\lambda_\beta-1)^2} \sim  \frac{1}{c^2} \exp(\beta H_\infty^0). 
\end{equation}
The proof of all the equivalents~\eqref{equation:selectionCase0} is now complete.
\end{proof}

\section{The nonselection case}
\label{section:nonSelectionCase}

We construct an example of Lipschitz double-well type potential  satisfying $H_\infty^0=H_\infty^1=0$ that produces a nonconvergent family of 
Gibbs measure as the temperature goes to zero. Notice that any symmetric example, $H_n^0=H_n^1$,  $\forall \, n\geq1$, provides a family of symmetric Gibbs measures $\{\mu_ \beta\}$ that converges to $\frac{1}{2}\delta_{0^\infty}+\frac{1}{2}\delta_{1^\infty}$. We show that the subclass of double-well type potentials  is rich enough to break  the symmetry in an alternated way. Notice also that $H$ is necessarily reduced in order to obtain the nonselection case.

The two fixed points $0^\infty$, $1^\infty$ are connected by two heteroclinic orbits, 
$\{0^n1^\infty\}_{n\geq1}$ and $\{1^n0^\infty\}_{n\geq 1}$. The oscillation between the two minimizing measures $\delta_{0^\infty}$ and $\delta_{1^\infty}$ are obtained by choosing a symmetric potential $H$, where both $\{H_n^0\}_{n\geq1}$ and $\{H_n^1\}_{n\geq1}$ are nonincreasing and converge to zero. The level sets of $H$ alternate as in figure \ref{figure:non_selection_case} and are chosen according to the following rules that are similar to the rules in \cite{VanEnterRusze2007l}.

\begin{figure}[hbt]
\centering
\includegraphics[width=0.98\textwidth]{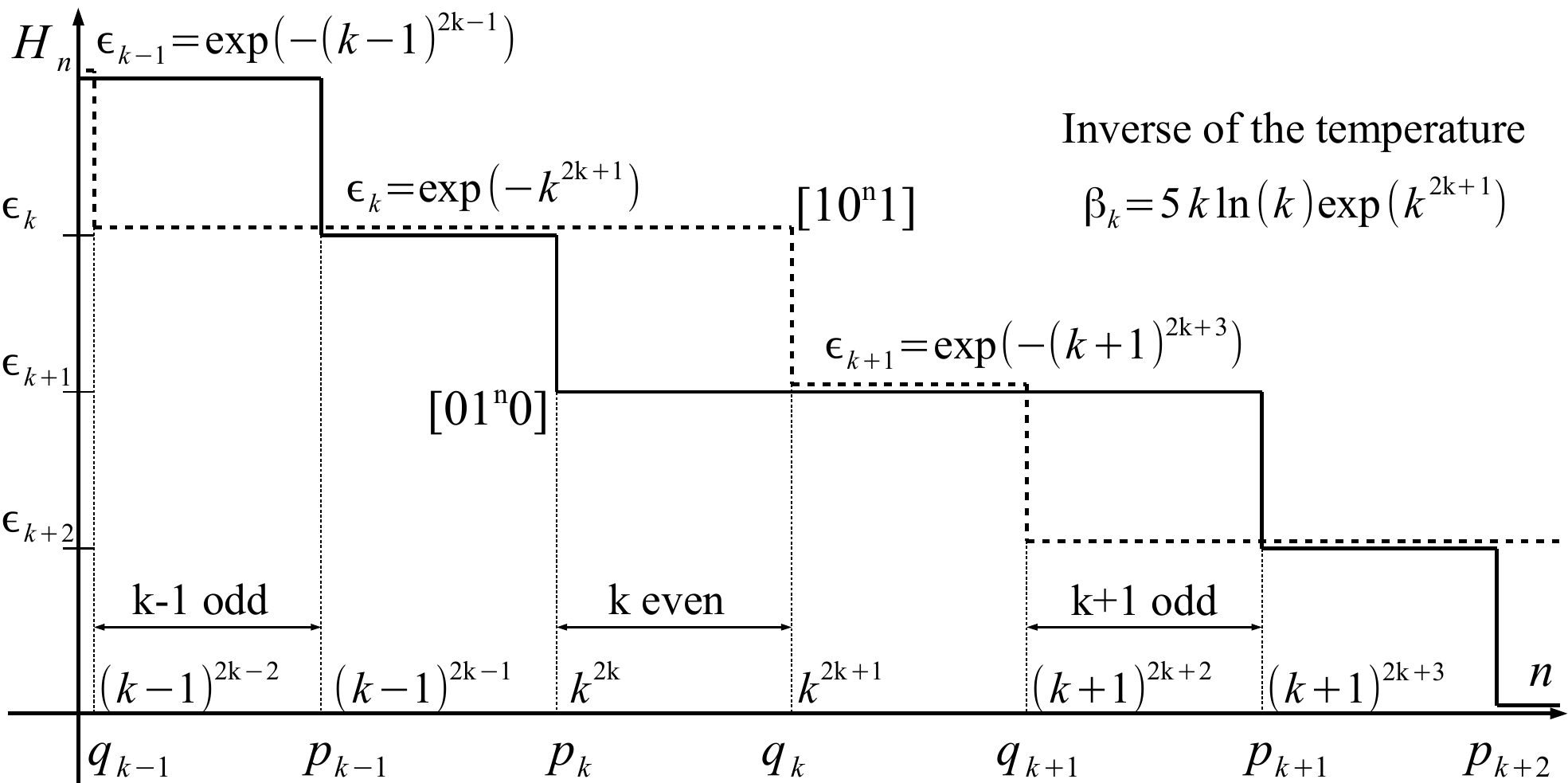}
\caption{\footnotesize{\textbf{The nonselection case for a Lipschitz example.}  The level sets satisfy $H = \epsilon_k = \exp(-k^{2k+1})$  on  $[01^n0]$ for every  $p_{k-1} < n \leq p_k$ and on  $[10^n1]$ for every  $q_{k-1} < n \leq q_k$. If $k$ is even, $p_k = k^{2k}$ and $q_k = k^{2k+1}$. If $k$ is odd, $p_k = k^{2k+1}$ and $q_k = k^{2k}$.}}
\label{figure:non_selection_case}
\end{figure}

\medskip
-- {\it Rule 1.} We choose two increasing sequences $\{p_k\}_{k\geq0}$ and $\{q_k\}_{k\geq0}$ which alternate according to the parity of the index $k$:
\begin{gather*}
1 \le p_0 < q_0< q_1 < p_1 < p_2 < q_2 < q_3  < p_3 < \ldots, \\
p_{2l} < q_{2l} < q_{2l+1} < p_{2l+1} < p_{2l+2} < q_{2l+2} < \ldots
\end{gather*}

-- {\it Rule 2.} We choose a decreasing sequence  $\{\epsilon_k\}_{k\geq0}$ of positive numbers which  goes to zero. We choose $H$ so that a level set of $H$ corresponds to a union of cylinders $[01^n0]$ (respectively $[10^n1]$)  over  $n \in \{p_{k-1}+1, \ldots,p_k\}$ (respectively over  $n \in \{q_{k-1}+1, \ldots, q_k\}$). By convention $p_{-1}=q_{-1}=0$, and
\begin{align*}
H_n^0 := \epsilon_k, \ \forall \, p_{k-1} < n \leq p_k, \qquad \ H_n^1 := \epsilon_k, \ \forall \, q_{k-1} < n \leq q_k.
\end{align*}
The contribution of the potential $H_n^0$ (respectively $H_n^1)$ exhibits a large drop at the level $p_k$ (respectively $q_k$): 
\begin{align*}
\forall \, n\leq p_k, \ H_n^0 \geq \epsilon_k, \quad \forall \, n \geq p_k+1, \ H_n^0 \leq \epsilon_{k+1}, \\
\forall \, n\leq q_k, \ H_n^1 \geq \epsilon_k, \quad \forall \, n \geq q_k+1, \ H_n^1 \leq \epsilon_{k+1}.
\end{align*}

-- {\it Rule 3.} We choose a decreasing sequence of temperatures $\beta_k^{-1}\to0$ which forces the Gibbs measure to give larger mass to either $[0]$ for an even index or $[1]$ for an odd index. The only constraints on $\{p_k\}$, $\{q_k\}$, $\{\epsilon_k\}$ and $\{\beta_k\}$ we use are:
\begin{gather*}
\lim_{k\to+\infty} p_k^2\exp(-\beta_k\epsilon_k)=0, \quad \lim_{k\to+\infty} q_k^2\exp(-\beta_k\epsilon_k)=0, \\
\lim_{k\to+\infty} \beta_k\epsilon_{k+1}=0, \quad 
\lim_{k\to+\infty} \frac{q_{2k}}{p_{2k}} = +\infty, \quad \lim_{k\to+\infty}    \frac{p_{2k+1}}{q_{2k+1}} = +\infty, \\
\sum_{k\geq1} (p_k-p_{k-1})\exp(-\epsilon_k) < + \infty, \quad \sum_{k\geq1} (q_k - q_{k-1})\exp(-\epsilon_k) < + \infty.
\end{gather*}
The last two conditions ensure the summability of the variation.

The three previous rules enable us to say that, at the temperature $\beta_k^{-1}$, for $k$ even or odd, the system is mainly governed by a system having a potential $\tilde H$ equal to zero on $[00] \cup [01^{p_k+1}] \cup [11] \cup [10^{q_k+1}]$ (thanks to  $\epsilon_{k+1} \ll \epsilon_{k}$), and positive elsewhere.

\begin{proof}[Proof of item \ref{item:mainTheorem_2} of theorem \ref{theorem:main}]
Let  $k$ be even. The other case is similar. To simplify the notations, we write $p=p_k$, $q=q_k$,  and $\lambda = \lambda_{\beta_k}$.  Remember the \emph{a priori} estimate $\lambda\leq 2$.

{\it Part 1.} We rewrite $F_\beta^0(\lambda)$ as if the energy $H_n^0$ where negligible for $n>p$. Then 
\begin{equation}  \label{equation:mainProoof_0}
 F_\beta^0(\lambda) = \frac{1}{\lambda^p(\lambda-1)} \big(\alpha_0 + \lambda^p(\lambda-1) \theta_0 \big),
\end{equation}
where
\[
\alpha_0 := \lambda^p(\lambda-1) \sum_{n\geq p+1} \frac{1}{\lambda^n} \exp(-\beta_k H_n^0), \ \text{ and } \ \theta_0 := \sum_{n=1}^p \frac{1}{\lambda^n} \exp(-\beta_k H_n^0).
\]
As $H_n^0 \leq \epsilon_{k+1}$ for $n \geq p+1$ and $H_n^0 \geq \epsilon_k$ for $n\leq p$, we obtain
\begin{gather*}
\exp(-\beta_k \epsilon_{k+1}) \leq \alpha_0 \leq 1, \quad \theta_0 \leq p \exp(-\beta_k \epsilon_k).
\end{gather*}
Rule 3 implies $\alpha_0 \to 1$ and $\theta_0 \to 0$ as $k\to +\infty$. Similarly
\begin{equation}  \label{equation:mainProoof_1/2}
 F_\beta^1(\lambda) = \frac{1}{\lambda^q(\lambda-1)} \big( \alpha_1 + \lambda^q(\lambda-1) \theta_1 \big),
\end{equation}  
with 
$$ \alpha_1 := \lambda^q(\lambda-1) \sum_{n\geq q+1} \frac{1}{\lambda^n} \exp(-\beta_k H_n^1), \ \text{ and } \
\theta_1 := \sum_{n=1}^q \frac{1}{\lambda^n} \exp(-\beta_k H_n^1). $$
As  $H_n^1 \leq \epsilon_{k+1}$ for $n \geq q+1$ and $H_n^1 \geq \epsilon_k$ for $n\leq q$, the third rule also implies   $\alpha_1 \to 1$ and $\theta_1 \to 0$ as $k\to +\infty$. As $F_\beta^0(\lambda)F_\beta^1(\lambda)=1$, we have
\begin{gather*}
\lambda^{p+q}(\lambda-1)^2 = \big[\alpha_0 + \lambda^p(\lambda-1) \theta_0 \big] \big[\alpha_1 + \lambda^q(\lambda-1) \theta_1 \big] := \delta^2.
\end{gather*}

{\it Part 2.} We show that  $\delta \to 1$ as $k\to+\infty$. Let  $N := \frac{p+q}{2}$.
We first observe that, for $k$ large enough,  $\lambda^N\geq e$. If not, 
\begin{equation} \label{equation:mainProoof_1}
\lambda - 1 \geq \delta e^{-1} \geq e^{-1}\sqrt{\alpha_0\alpha_1}.
\end{equation}
On the one side $\lambda -1 \to 0$, on the other side $\alpha_0\alpha_1 \to 1$; we get a contradiction. 
We next observe that $\lambda-1 \geq \frac{1}{N}$. Indeed
\begin{gather}
\lambda = 1 + \frac{\delta}{\lambda^N}, \quad \ln(\lambda) \leq \frac{\delta}{\lambda^N}, \quad 1 \leq N\ln(\lambda) \leq \frac{N\delta}{\lambda^N}, \quad \lambda^N \leq N\delta, \label{equation:mainProoof_2}
\end{gather}
and from the equation $\lambda^N(\lambda-1) = \delta$, we finally obtain $\lambda-1 \geq \frac{1}{N}$. 
We rewrite the two terms $\lambda^p(\lambda-1)$ and $\lambda^q(\lambda-1)$ as
\begin{align*}
\lambda^p(\lambda-1) &= (\lambda^N)^{p/N}(\lambda-1) = \big[ \lambda^N(\lambda-1) \big]^{p/N} (\lambda-1)^{1-p/N} \\
&=\delta^{p/N} (\lambda-1)^{(q-p)/(q+p)} \leq \delta^{p/N}, \\
\lambda^q(\lambda-1) &= (\lambda^N)^{q/N}(\lambda-1) = \big[ \lambda^N(\lambda-1) \big]^{q/N} (\lambda-1)^{1-q/N} \\
&= \delta^{q/N} (\lambda-1)^{-(q-p)/(q+p)} \leq \delta^{q/N} (\lambda-1)^{-1} \leq q \delta^{q/N}.
\end{align*}
Therefore, we have
\begin{align*}
\delta^2 &\leq \big[ \alpha_0 + \delta^{p/N} \theta_0 \big] \big[ \alpha_1 + q\delta^{q/N} \theta_1 \big] \\
&= \alpha_0\alpha_1 + \alpha_0\theta_1 q \delta^{q/N} + \alpha_1\theta_0\delta^{p/N} + \theta_0\theta_1 q \delta^2.
\end{align*}
Using $\delta^{p/N} \leq 1+ \delta^2$ and $\delta^{q/N} \leq 1+\delta^2$, we have
\[
\alpha_0 \alpha_1 \leq \delta^2 \leq 
\frac{\alpha_0 \alpha_1 + (\alpha_0 q \theta_1 + \alpha_1 \theta_0)}{1-(\alpha_0 q \theta_1 + \alpha_1 \theta_0 + \theta_0 q\theta_1)}.
\]
Since $q\theta_1 \leq q^2\exp(-\beta_k\epsilon_k) \to 0$ and $\theta_0 \to 0$ as $k\to+\infty$, we obtain $\delta\to1$. 

{\it Part 3.} We first prove that $q(\lambda-1) \to +\infty$. Since $N < q$, it is enough to show $N(\lambda-1) \to +\infty$. 
Indeed, for every $C \ge 1$,  for $k$ sufficiently large, $\lambda^N \geq \exp(C)$ as in \eqref{equation:mainProoof_1}. 
Using the same estimates as in \eqref{equation:mainProoof_2}, we have 
\[
C\lambda^{N} \leq N \delta \ \ \text{ and }  \ \  N(\lambda-1) \geq C.
\] 
Therefore, from the estimates of part 2, we see that 
\begin{align*}
&\frac{\lambda^p(\lambda-1)^2}{p(\lambda-1)+\lambda} \leq \frac{\lambda^p(\lambda-1)}{p} \leq \frac{\delta^{p/N}}{p} \leq \frac{1+\delta^2}{p} \to 0, \\
&\frac{\lambda^q(\lambda-1)^2}{q(\lambda-1)+\lambda} \leq \frac{\lambda^q(\lambda-1)}{q} \leq \frac{\delta^{q/N}}{q(\lambda-1)} \leq \frac{1+\delta^2}{q(\lambda-1)} \to 0.
\end{align*}

{\it Part 4.} We decompose $\tilde F_\beta^0(\lambda)$ as before
\begin{gather} \label{equation:mainProoof_4}
\tilde F_\beta^0(\lambda) = \frac{p(\lambda-1)+\lambda}{\lambda^p(\lambda-1)^2} \Big( \tilde \alpha_0 + \frac{\lambda^p(\lambda-1)^2}{p(\lambda-1)+\lambda} \tilde \theta_0 \Big),
\end{gather}
where 
\begin{gather*}
\exp(-\beta_k \epsilon_{k+1}) \leq \tilde \alpha_0 := \frac{\lambda^p(\lambda-1)^2}{p(\lambda-1)+\lambda} \sum_{n\geq p+1} \frac{n}{\lambda^n} \exp(-\beta_k H_n^0) \leq 1,  \\ 
\text{and}\quad \tilde\theta_0 := \sum_{n=1}^{p}  \frac{n}{\lambda^n} \exp(-\beta_k H_n^0)  \leq p^2 \exp(-\beta_k \epsilon_{k}).
\end{gather*}
Then $\tilde \alpha_0 \to 1$ and $\tilde\theta_0 \to 0$. Similar estimates are obtained for $\tilde F_\beta^1(\lambda)$.

{\it Part 5.} We may now conclude the proof. Since $\lambda^p(\lambda-1)/p \to 0$, $\lambda^q(\lambda-1)/q \to 0$, $p\theta_0 \to 0$
and $q\theta_1 \to 0$, equations~\eqref{equation:mainProoof_0} and~\eqref{equation:mainProoof_1/2} imply
\[
F_\beta^0(\lambda) \sim \frac{1}{\lambda^p(\lambda-1)} \quad \text{and} \quad F_\beta^1(\lambda) \sim \frac{1}{\lambda^q(\lambda-1)}.
\]
As $\lambda^p(\lambda-1)^2/(p(\lambda-1)+\lambda) \to 0$ and $\lambda^q(\lambda-1)^2/(q(\lambda-1)+\lambda) \to 0$,  
equation~\eqref{equation:mainProoof_4} and a similar expression for $ \tilde F_\beta^1(\lambda) $ provide
\[
\tilde F_\beta^0(\lambda) \sim \frac{p(\lambda-1)+\lambda}{\lambda^p(\lambda-1)^2} \quad\text{and} \quad \tilde F_\beta^1(\lambda) \sim \frac{q(\lambda-1)+\lambda}{\lambda^q(\lambda-1)^2}.
\]
Item \ref{corollary:estimatesForGibbsMeasure_5} of Corollary \ref{corollary:estimatesForGibbsMeasure} thus gives
\begin{gather*}
\frac{\mu_\beta[0]}{\mu_\beta[1]} = \frac{F_\beta^0(\lambda)}{F_\beta^1(\lambda)} \frac{\tilde F_\beta^1(\lambda)}{\tilde F_\beta^0(\lambda)} \sim \frac{q(\lambda-1)+\lambda}{p(\lambda-1)+\lambda} \geq  \min \Big\{ \frac{q}{2p}, \frac{q(\lambda-1)}{2\lambda} \Big\} \to +\infty.
\end{gather*}
As a matter of fact, rule 3 asks $\lim_{l\to+\infty} \frac{q_{2l}}{p_{2l}} = +\infty$.
Hence, $\mu_{\beta_{2l}} \to \delta_{0^\infty}$.
\end{proof}

\footnotesize
\end{document}